\documentclass[brochure,english,12pt]{smfbourbaki}

\usepackage[dvipsnames]{xcolor}
\usepackage[T1]{fontenc}
\usepackage{lmodern,amssymb,bm,tikz,mathtools,hhline,tabu,caption,enumitem}
\usetikzlibrary{backgrounds,calc}
\tikzset{vert/.style={draw, fill=black, circle, inner sep=2pt}}
\usepackage[theorems,breakable,skins]{tcolorbox}

\usepackage{babel}

\usepackage[utf8]{inputenc}

\usepackage[colorlinks=true, linkcolor=blue, citecolor=red,
urlcolor=blue]{hyperref}
\usepackage[capitalize]{cleveref}
\usepackage{todonotes}

\newcommand\ab[1]{\lvert#1\rvert}
\newcommand\wt[1]{\widetilde{#1}}
\newcommand\wh[1]{\widehat{#1}}
\newcommand\inn[1]{\langle #1 \rangle}
\newcommand\R{\mathbb{R}}
\newcommand\E{\mathbb{E}}
\newcommand\cE{\mathcal{E}}
\DeclareMathOperator\pr{Pr}
\newcommand\dd{\,\mathrm{d}}

\newtheorem{theorem}{Theorem}[section]
\newtheorem{lemma}[theorem]{Lemma}
\newtheorem{proposition}[theorem]{Proposition}

\newtheorem{conjecture}[theorem]{Conjecture}

\newtheorem{assumption}[theorem]{Assumption}
\theoremstyle{remark}
\newtheorem{definition}[theorem]{Definition}

\newcommand\fixind{\leavevmode \vspace{-\baselineskip}}

\newtcbtheorem[crefname={Algorithm}{Algorithms}, use counter*=theorem]{algorithm}{Algorithm}{enhanced jigsaw, colframe=black, colback=white, colbacktitle=white, coltitle=black, halign title = center, sharp corners=all, boxrule=.7pt, titlerule=-1pt, toptitle=2pt, top=7pt, description delimiters parenthesis, separator sign none, fonttitle=\scshape, description font=\normalfont, left=0pt, breakable}{alg}

\crefname{equation}{equation}{equations}
\crefname{lemma}{Lemma}{Lemmas}
\crefname{proposition}{Proposition}{Propositions}
\crefname{claim}{Claim}{Claims}
\crefname{theorem}{Theorem}{Theorems}
\crefname{conjecture}{Conjecture}{Conjectures}
\crefname{figure}{Figure}{Figures}
\crefname{assumption}{Assumption}{Assumptions}

\numberwithin{equation}{section}
\numberwithin{table}{section}

\newlist{lemenum}{enumerate}{1}
\setlist[lemenum]{label=(\alph*), ref=\thelemma(\alph*)}
\crefalias{lemenumi}{lemma}

\let\geq\geqslant
\let\leq\leqslant

\usepackage[stretch=10,shrink=10,step=2,kerning=true,protrusion=true,expansion=true,final]{microtype}

\usepackage[
backend=biber,
style=authoryear, 
citestyle=authoryear-comp,
maxnames=7,
maxbibnames=8,
uniquename=false,
sortcites=false 
]{biblatex}
\usepackage{csquotes}

\DefineBibliographyExtras{french}{\restorecommand\mkbibnamefamily}

\DeclareDelimFormat{nameyeardelim}{\addcomma\space}

\DeclareNameAlias{sortname}{given-family}

\renewcommand{\bibnamedash}{\leavevmode\raise3pt\hbox to3em{\hrulefill}\space}

\AtEveryBibitem{%
  \clearfield{issn} 
  \clearfield{isbn} 
  \clearfield{doi} 
  \clearlist{language} 
  \ifentrytype{online}{}{
  \ifentrytype{unpublished}{}{
    \clearfield{url}
  }
  }
}

\renewbibmacro{in:}{%
    \ifentrytype{article}{}{\printtext{\bibstring{in}\intitlepunct}}}

\DeclareFieldFormat[article,periodical,inreference]{number}{\mkbibparens{#1}}
\DeclareFieldFormat[article,periodical,inreference]{volume}{\mkbibbold{#1}}
\renewbibmacro*{volume+number+eid}{%
    \printfield{volume}%
    \setunit*{\addthinspace}
    \printfield{number}%
    \setunit{\addcomma\space}%
    \printfield{eid}}

\DeclareFieldFormat[article,inbook,incollection]{title}{\enquote{#1}\addcomma} 

\date{Novembre 2024}
\bbkannee{77\textsuperscript{e} année, 2024--2025}  
\bbknumero{1230}                                      

\title{Upper bounds on diagonal Ramsey numbers}
\subtitle{after Campos, Griffiths, Morris, and Sahasrabudhe}

\author{Yuval Wigderson}

\address{Institute for Theoretical Studies,\\ ETH Z\"urich}

\email{yuval.wigderson@eth-its.ethz.ch}

\newcounter{savefootnote}
\newcounter{symfootnote}
\newcommand{\symfootnote}[1]{%
   \setcounter{savefootnote}{\value{footnote}}%
   \setcounter{footnote}{\value{symfootnote}}%
   \ifnum\value{footnote}>8\setcounter{footnote}{0}\fi%
   \let\oldthefootnote=\thefootnote%
   \renewcommand{\thefootnote}{\fnsymbol{footnote}}%
   \footnote{#1}%
   \let\thefootnote=\oldthefootnote%
   \setcounter{symfootnote}{\value{footnote}}%
   \setcounter{footnote}{\value{savefootnote}}%
}
\stepcounter{symfootnote}
\newcommand\dangerfoot{\hyperref[foot:danger]{\textsuperscript{(\textdagger)}}}

\begin{document}

\maketitle

\section{Introduction}\label{sec:intro}
Ramsey theory is a branch of combinatorics that studies order and disorder. The underlying mantra of the field, as articulated by Theodore Motzkin, is that \enquote{complete disorder is impossible}---any sufficiently large system must have a large, highly structured subsystem. The prototypical example of a Ramsey-theoretic statement is \emph{Ramsey's theorem}, from which the field derives its name. 
\begin{theorem}[\cite{MR1576401}]\label{thm:ramsey}
	For every integer $k\geq 2$, there exists some positive integer $N$ such that any two-coloring of the edges of the complete graph\footnote{Recall that the complete graph $K_N$ has $N$ vertices, and all of the $\binom N2$ possible edges are present.} $K_N$ contains a monochromatic $K_k$.
\end{theorem}
In other words, no matter how we assign the edges of $K_N$ a color, say red or blue, we can always find $k$ vertices such that all edges between them receive the same color. That is, any such coloring, no matter how unstructured, contains a highly structured subcoloring.  Even this simple statement has some remarkable consequences. For example, \textcite{Schur1917} used \cref{thm:ramsey}\footnote{Alert readers may note that Schur's result precedes Ramsey's by more than a decade. In fact, Schur proved a closely related lemma, which one can now recognize as a consequence of \cref{thm:ramsey}, and derived his theorem from that lemma.} to prove that for all sufficiently large primes $p$, there exist non-trivial solutions to the equation $x^n + y^n \equiv z^n \pmod p$, that is, that one cannot prove Fermat's last theorem via a local-global argument.

Connections and applications to other fields of mathematics have been an important feature of Ramsey theory from the very beginning. Ramsey himself had an application in mathematical logic in mind when he proved \cref{thm:ramsey} (indeed, his paper is titled \enquote{On a problem of formal logic}). The influential paper of \textcite{MR1556929}, which helped establish Ramsey theory as a central branch of combinatorics, is titled ``A combinatorial problem in geometry''; in it, they reproved \cref{thm:ramsey} in order to deduce a result on convex polygons among sets of points in Euclidean space.

Today, Ramsey-theoretic theorems and techniques are of fundamental importance in many different fields, including additive number theory, Banach space theory, discrete geometry, ergodic theory, group theory, and theoretical computer science. These are deep and rich connections, and are difficult to adequately summarize, so we refer to the book of \textcite{MR1044995}, to the survey of \textcite{MR3497267}, and to the lecture notes of \textcite{Wigderson_lecture_notes} for more in-depth introductions to the field.

For many applications, such as those of \textcite{Schur1917} in number theory, \textcite{MR1576401} in logic, and \textcite{MR1556929} in geometry mentioned above, qualitative statements such as \cref{thm:ramsey} suffice. However, much of the modern research in Ramsey theory is concerned with \emph{quantitative} statements: how large is the integer $N$ in \cref{thm:ramsey} as a function of $k$? Formally, we make the following definition.
\begin{definition}\label{def:ramsey number}
	The \emph{Ramsey number} $r(k)$ is the least integer $N$ such that every two-coloring of the edges of $K_N$ contains a monochromatic $K_k$.
\end{definition}
Before continuing with the discussion of what is known about the function $r(k)$, let us pause and ask why we should study such quantitative questions, when qualitative statements like \cref{thm:ramsey} are elegant and already suffice for many applications. There are several answers to this question. One answer is that for certain applications, especially in fields such as theoretical computer science (e.g.\ the lower bound of \textcite{MR785629} on monotone circuit complexity), qualitative statements are not sufficient, as the application itself is quantitative. A second answer is that a better quantitative understanding of Ramsey-theoretic results can yield new insights and new proofs of existing theorems. For example, recent breakthroughs on the quantitative aspects of the Ramsey-theoretic theorem of \textcite{MR0051853}, due to \textcite{2007.03528,MR4720301} (see also the expos\'e of \textcite{MR4576028}), imply that the primes contain infinitely many three-term arithmetic progressions. This result was first proved by \textcite{MR1513216}, and is a special case of the landmark result of \textcite{MR2415379}. However, in contrast to these earlier proofs, we now know that the primes contain infinitely many three-term arithmetic progressions simply because \emph{there are many prime numbers}. That is, the quantitative improvements yielded a new proof of this theorem, using essentially no properties of the primes other than their density. Finally, and no less importantly, a third reason for studying such quantitative questions is that doing so can reveal a world of deep and beautiful mathematics.

With that said, let us turn to the quantitative aspects of \cref{thm:ramsey}, that is, to the determination of the function $r(k)$ from \cref{def:ramsey number}. The exact value of $r(k)$ is only known for $k \leq 4$, and it currently seems completely hopeless\footnote{The following famous anecdote was reported by \textcite{MR1249485}:
``Erd\H os asks us to imagine an alien force, vastly more powerful than us, landing on Earth and demanding the value of $r(5)$ or they will destroy our planet. In that case, he claims, we should marshal all our computers and all our mathematicians and attempt to find the value. But suppose, instead, that they ask for $r(6)$. In that case, he believes, we should attempt to destroy the aliens.'' Indeed, results of \textcite{MR982871} and \textcite{2409.15709} show that $r(5)$ lies takes on one of the four values $\{43,44,45,46\}$, but we remain very far from knowing the value of $r(6)$.} to obtain an exact formula for $r(k)$, so let us content ourselves with asymptotic bounds as $k \to \infty$. Essentially every proof of \cref{thm:ramsey} yields (at least implicitly) an upper bound on $r(k)$, by proving the existence of \emph{some} integer $N$. The original proof of \textcite{MR1576401} gave a bound of $r(k) \leq k!$, but Ramsey wrote \enquote{I have little doubt that [this upper bound is]
far larger than is necessary}. Indeed, a few years later, \textcite{MR1556929} proved the following stronger bound.
\begin{theorem}[\cite{MR1556929}]\label{thm:ES}
	$r(k) \leq 4^k$ for every $k \geq 2$.
\end{theorem}
For about a decade, it was believed that this bound was also far larger than is necessary, namely that $r(k)$ should grow subexponentially as a function of $k$. However, \textcite{MR0019911} dispelled this belief by proving\footnote{Lower bounds on Ramsey numbers are somewhat beyond the scope of this expos\'e, so we will not discuss the proof of \cref{thm:exp LB} in detail. However, it would be remiss not to mention that this beautiful proof is extraordinarily influential, and is the origin of the \emph{probabilistic method}, an extremely powerful technique in modern combinatorics.} an exponential lower bound.
\begin{theorem}[\cite{MR0019911}]\label{thm:exp LB}
	$r(k)\geq {\sqrt {2}}^{\,k}$ for every $k \geq 2$.
\end{theorem}
After this breakthrough, progress stalled for 75 years. There were a number of improvements to these bounds over the years, including important results of \textcite{MR366726,MR0905278,MR0968746,MR2552114,MR4548417}, but all of these improvements only affected the lower-order terms, and did not improve either of the exponential constants $\sqrt 2$ and $4$. This impasse finally ended with a breakthrough of \textcite{2303.09521}.
\begin{theorem}[\cite{2303.09521}]\label{thm:3.9999}
	There exists a constant $\delta>0$ such that $r(k) \leq (4-\delta)^k$ for all $k \geq 2$.
	Concretely, $r(k) \leq 3.993^k$ for all sufficiently large $k$.
\end{theorem}
The exact constant $3.993$ is not particularly important, and a more careful analysis of the same proof yields a slightly better bound\footnote{More recently, \textcite{2407.19026} recast the proof of \cref{thm:3.9999} in a different language, which allowed them to optimize the technique and obtain a much stronger bound of $r(k) \leq 3.8^k$ for sufficiently large $k$.}. The important thing about this result is that it is the first result, after almost 90 years of intense study, to break the barrier of $4^k$. 

The new tool introduced by \textcite{2303.09521} is the so-called \emph{book algorithm}, an elementary but ingenious technique for finding monochromatic \emph{book graphs} in colorings of $K_N$. As we will shortly discuss, a book graph is a basic graph-theoretic object, whose study turns out to be closely connected to the study of Ramsey numbers. Every known proof of \cref{thm:ramsey} uses, implicitly or explicitly, monochromatic book graphs. 

As we will see, the proof of \textcite{2303.09521} is fairly ad hoc, and relies on the verification of certain complicated numerical inequalities. More recently, however, a new, more conceptual proof of \cref{thm:3.9999}\footnote{Moreover, \textcite{2410.17197} were able to prove a more general theorem, which gives a new upper bound on Ramsey numbers in any number of colors. For simplicity, however, we remain with the two-color version of the problem throughout this expos\'e.} was found by \textcite{2410.17197}. They introduced a modification of the book algorithm, but their crucial new input is a purely geometric lemma, concerning the correlations of probability distributions in high-dimensional Euclidean space. While the proof of \textcite{2303.09521} and \textcite{2410.17197} have many features in common, they differ in key ways, and we will sketch both proofs.

The rest of this expos\'e is dedicated to discussing these two proofs of \cref{thm:3.9999}, and is organized as follows. We begin in \cref{sec:background} with a proof of \cref{thm:ES}, in the course of which we introduce book graphs as well as several of the key ideas that go into the proof of \cref{thm:3.9999}. In \cref{sec:book alg}, we introduce and analyze the book algorithm of \textcite{2303.09521}, and will then \emph{fail} to prove \cref{thm:3.9999}. Luckily, we will rescue the argument and complete the proof in \cref{sec:rescue} by introducing two additional ingredients. In \cref{sec:symmetric alg} we introduce and analyze the symmetric book algorithm of \textcite{2410.17197}, and use it to give another proof of \cref{thm:3.9999}. The key new lemma introduced by \textcite{2410.17197}, and its geometric proof, are discussed in \cref{sec:geometry}. We end in \cref{sec:epilogue} with an epilogue, discussing the use of book graphs in the original proof of \textcite{MR1576401} of \cref{thm:ramsey}, as well as how our understanding of book graphs and Ramsey theory has developed over the subsequent 95 years.

\paragraph*{Acknowledgments}
An early version of this expos\'e was written for the lecture notes of a Ramsey theory course that I taught at ETH in Spring 2024; I am grateful to all of the students in the course for their interest and insights. I would also like to thank Nicolas Bourbaki, Marcelo Campos, Xiaoyu He, Zach Hunter, Eoin Hurley, Greg Kuperberg, Vivian Kuperberg, and Wojciech Samotij for many helpful discussions and comments on earlier drafts. I am supported by Dr.\ Max R\"{o}ssler, the Walter Haefner Foundation, and the ETH Z\"{u}rich Foundation.

\section{The Erd\H os--Szekeres theorem and algorithm}\label{sec:background}
In this section, we prove \cref{thm:ES} (and thus \cref{thm:ramsey}). This proof is elegant and interesting in its own right, and additionally it contains within it several of the important ideas used in the proof of \cref{thm:3.9999}. We will actually see three different proofs (or, more precisely, three different ways of viewing the same proof) of \cref{thm:ES}, in each of the next three subsections. Each proof will help introduce some of the key ideas that go into the proof of \cref{thm:3.9999}.
\subsection{Off-diagonal Ramsey numbers}
We begin with the original proof of \textcite{MR1556929}. 
Before proceeding with the proof, we generalize the notion of Ramsey numbers from \cref{def:ramsey number}.
Here and throughout, we denote by $V(K_N)$ and $E(K_N)$ the vertex set and edge set, respectively, of the complete graph $K_N$.
\begin{definition}\label{def:off diagonal}
	Given integers $k,\ell\geq 2$, the \emph{off-diagonal Ramsey number} $r(k,\ell)$ is the least integer $N$ such that every two-coloring of $E(K_N)$ with colors red and blue contains a red $K_k$ or a blue $K_\ell$.
\end{definition}
Note that $r(k,\ell)=r(\ell,k)$ as the colors play symmetric roles, and that $r(k) = r(k,k)$. The quantity $r(k)$ is often called the \emph{diagonal Ramsey number}.

With this terminology, we can prove \cref{thm:ES}. In fact, we will prove the following more precise result.
\begin{theorem}[\cite{MR1556929}]\label{thm:ES binomial}
	For all integers $k,\ell\geq 2$, we have
	\[
		r(k,\ell) \leq \binom{k+\ell-2}{k-1}.
	\]
	In particular,
	\[
		r(k) \leq \binom{2k-2}{k-1} < 4^k.
	\]
\end{theorem}
\begin{proof}
	We proceed by induction on $k+\ell$, with the base case $\min\{k,\ell\}=2$ being trivial. For the inductive step, the key claim is that the following inequality holds:
	\begin{equation}\label{eq:ES recursion}
		r(k,\ell) \leq r(k-1,\ell) + r(k,\ell-1).
	\end{equation}
	To prove \eqref{eq:ES recursion}, fix a red/blue coloring of $E(K_N)$, where $N = r(k-1,\ell)+r(k,\ell-1)$, and fix some vertex $v \in V(K_N)$. Suppose for the moment that $v$ is incident to at least $r(k-1,\ell)$ red edges, and let $R$ denote the set of endpoints of these red edges. By definition, as $\ab R \geq r(k-1,\ell)$, we know that $R$ contains a red $K_{k-1}$ or a blue $K_\ell$. In the latter case we have found a blue $K_\ell$ (so we are done), and in the former case we can add $v$ to this red $K_{k-1}$ to obtain a red $K_k$ (and we are again done).

	So we may assume that $v$ is incident to fewer than $r(k-1,\ell)$ red edges. By the exact same argument, just interchanging the roles of the colors, we may assume that $v$ is incident to fewer than $r(k,\ell-1)$ blue edges. But then the total number of edges incident to $v$ is at most
	\[
		(r(k-1,\ell)-1) + (r(k,\ell-1)-1) = N-2,
	\]
	which is impossible, as $v$ is adjacent to all $N-1$ other vertices. This is a contradiction, proving \eqref{eq:ES recursion}.

	We can now complete the induction. By \eqref{eq:ES recursion} and the inductive hypothesis, we find that
	\begin{align*}
		r(k,\ell)&\leq r(k-1,\ell) + r(k,\ell-1) \\
		&\leq \binom{(k-1)+\ell-2}{(k-1)-1} + \binom{k+(\ell-1)-2}{k-1} \\
		&= \binom{k+\ell-2}{k-1},
	\end{align*}
	where the final equality is Pascal's identity for binomial coefficients.
\end{proof}

\subsection{Enter the book}\label{sec:enter book}
\begin{definition}
	Let $t,m$ be positive integers. 
	The \emph{book graph} $B_{t,m}$ consists of a copy of $K_t$, plus $m$ additional vertices which are adjacent to all vertices of the $K_t$, but not adjacent to one another. Equivalently, $B_{t,m}$ is obtained from the complete bipartite graph $K_{t,m}$ by adding in all the $\binom t2$ possible edges in the side of size $t$. Equivalently, $B_{t,m}$ consists of $m$ copies of $K_{t+1}$ which are glued along a common $K_t$.
\end{definition}
Note that two important special cases are $m=1$, where $B_{t,1}$ is simply the complete graph $K_{t+1}$, and $t=1$, where $B_{1,m}$ is simply the star graph $K_{1,m}$, consisting of one vertex joined to $m$ others (and no other edges).
The ``book'' terminology comes from the case $t=2$, in which case $B_{2,m}$ consists of $m$ triangles sharing an edge, which looks, to some extent, like a book with $m$ triangular pages. Continuing this analogy, the $K_t$ in $B_{t,m}$ is called the \emph{spine}, and the $m$ additional vertices of $B_{t,m}$ are called the \emph{pages}. We will often denote a book as a pair of sets $(A,Y)$, where $A$ is the spine and $Y$ comprises the pages.

The reason book graphs are important in the study of Ramsey numbers comes down to the following simple observation.

\begin{lemma}\label{lem:book to clique}
	Suppose that a two-coloring of $E(K_N)$ contains a monochromatic red copy of $B_{t,m}$, where $m \geq r(k-t, \ell)$. Then this coloring contains a red $K_k$ or a blue $K_\ell$.
\end{lemma}
\begin{proof}
	Let $A$ be the spine of the book, and let $Y$ be its pages. By assumption, $\ab Y = m \geq r(k-t,\ell)$, so $Y$ contains a blue $K_\ell$ or a red $K_{k-t}$. In the former case we are done, and in the latter case, we may add $A$ to the red $K_{k-t}$ to obtain a red $K_k$.
\end{proof}
This proof should look familiar---we have already encountered the same idea in the proof of \cref{thm:ES binomial}, where we implicitly used the $t=1$ case of \cref{lem:book to clique}. Indeed, in that proof, we showed that if a coloring contains a red star with $r(k-1,\ell)$ leaves, then it contains a red $K_k$ or a blue $K_\ell$. The only new idea in \cref{lem:book to clique} is that we don't need to consider a single vertex (i.e.\ the case $t=1$), but may take an arbitrary book. 

Although the idea of \cref{lem:book to clique} basically goes back to the work of \textcite{MR1556929}, it was first formulated in essentially this language by \textcite{MR679211}, who used \cref{lem:book to clique} to propose a natural approach to improving the upper bounds on $r(k)$. Namely, if one can show that every two-coloring of $E(K_N)$ contains a monochromatic $B_{t,m}$, for some appropriate parameters $t$ and $m\geq r(k-t,k)$, then one can plug this into \cref{lem:book to clique} and conclude that $r(k) \leq N$. Again, this is essentially the approach we used in the proof of \cref{thm:ES binomial}, where a simple argument based on the pigeonhole principle showed that any coloring of $E(K_N)$ contains a large monochromatic star, that is, a monochromatic book with many pages and a spine of size $t=1$. The idea behind Thomason's program is that perhaps for larger values of $t$, more sophisticated arguments than the pigeonhole principle could yield stronger results, and improve the upper bounds on $r(k)$.

Thomason's idea has been quite successful. The three prior asymptotic improvements to \cref{thm:ES binomial}, due to \textcite{MR0968746,MR2552114,MR4548417}, all used this idea, roughly showing that if some two-coloring of $E(K_N)$ does \emph{not} contain a monochromatic $B_{t,m}$ (for some fixed $t,m$), then its structure must be such that the proof of \cref{thm:ES binomial} can be made more efficient. A more precise structural result along these lines is given by \textcite{MR4482092}. However, for fundamental technical reasons, none of these techniques seems capable of finding books with spine larger than $t=O(\log k)$, whereas in order to prove a result like \cref{thm:3.9999} in this way, one would want to take (say) $t=\frac{k}{1000}$. This was where the matter stood for a few years, until the work of \textcite{2303.09521}.

\subsection{The Erd\H os--Szekeres algorithm}\label{sec:ES algorithm}

One of the many new ingredients introduced by \textcite{2303.09521} is the following simple idea: rather than searching for some \emph{specific} book $B_{t,m}$, they define an exploration algorithm for finding \emph{some} book, and then prove that regardless of which book is found, the parameters involved are good enough to plug into \cref{lem:book to clique}. Although this idea is almost a triviality, this change of perspective is crucial for the proof of \cref{thm:3.9999}.

Before we define this exploration algorithm---which they termed the \emph{book algorithm}---let us first rephrase the proof of \cref{thm:ES} as an exploration algorithm, the \emph{Erd\H os--Szekeres algorithm}. 
Let us fix a two-coloring of $E(K_N)$. We assume that this coloring has no monochromatic $K_k$, and our goal is to eventually obtain a contradiction if $N$ is sufficiently large. For the moment we only seek to get a contradiction if $N > 4^k$, and thus reprove \cref{thm:ES}.

For a vertex $v \in V(K_N)$, we write $N_R(v)$ for the \emph{red neighborhood} of $v$, that is, the set of vertices $w \in V(K_N)$ such that the edge $vw$ is colored red. Similarly, $N_B(v)$ denotes the \emph{blue neighborhood} of $v$.

In the Erd\H os--Szekeres algorithm, we maintain three disjoint sets $A,B,X\subseteq V(K_N)$; the sets $A$ and $B$ will grow throughout the process, whereas $X$ will shrink. The key property we maintain is that $(A,X)$ is a red book, and $(B,X)$ is a blue book; that is, $A$ is completely red, $B$ is completely blue, all edges between $A$ and $X$ are red, and all edges between $B$ and $X$ are blue. To initialize the process, we set $A=B=\varnothing$ and $X=V(K_N)$. We now repeatedly run the following steps.

\begin{algorithm}{Erd\H os--Szekeres algorithm}{ES}
\fixind
\begin{enumerate}
	
	\item If $\ab X \leq 1$, $\ab A\geq k$, or $\ab B \geq k$, stop the process. \label{it:check sizes}
	\item Pick a vertex $v \in X$, and check whether $v$ has at least $\frac 12(\ab X-1)$ red neighbors in $X$. 
	
	\item  If yes, move $v$ to $A$ and shrink $X$ to the red neighborhood of $v$. That is, update $A \to A\cup \{v\}$ and $X \to X \cap N_R(v)$, and keep $B$ the same. Call this a \emph{red step}.
	\item If not, then $v$ has at least $\frac 12 (\ab X-1)$ blue neighbors in $X$. We now move $v$ to $B$, and shrink $X$ to the blue neighborhood of $v$. That is, we update $B \to B\cup \{v\}$ and $X \to X \cap N_B(v)$, and keep $A$ the same. Call this a \emph{blue step}.
	\item Return to step \ref{it:check sizes}.
\end{enumerate}
\end{algorithm}
By the way we update the sets, we certainly maintain the key property that $(A,X)$ and $(B,X)$ are red and blue books, respectively, throughout the entire process, since every time we add a vertex $v$ to $A$ (resp.\ $B$), we shrink $X$ to the red (resp.\ blue) neighborhood of $v$. 

Using \cref{alg:ES}, we can give an alternative proof of \cref{thm:ES}.
\begin{proof}[``Algorithmic'' proof of \cref{thm:ES}]
	Let $N = 4^k$, and fix a two-coloring of $E(K_N)$. Assume for contradiction that this coloring contains no monochromatic $K_k$. We now run \cref{alg:ES} until it terminates.

	Suppose first that the algorithm terminated because $\ab A \geq k$. Throughout the process, we maintain the property that all edges inside $A$ are red. Therefore, if $\ab A\geq k$ at the end of the process, we have found a monochromatic red $K_k$, a contradiction. Similarly, if $\ab B \geq k$ at the end of the process, we have found a blue $K_k$, another contradiction. We may thus assume that at the end of the process, we have $\ab A<k$ and $\ab B <k$.

	Therefore, the process can only end when $\ab X \leq 1$. The key observation now is that at every step of the process, we have
	\begin{equation}\label{eq:X size ES}
		\ab X \geq 2^{-\ab A - \ab B} N.
	\end{equation}
	Indeed, this certainly holds when the process begins, for then we have $\ab A = \ab B=0$ and $\ab X = N$. We can now check that it holds by induction: every time we do a red step, we increase $\ab A$ by $1$, and decrease $\ab X$ to at least\symfootnote{\label{foot:danger}Strictly speaking, we should write here $\frac 12 (\ab X-1)$, although the claimed bound \eqref{eq:X size ES} can also be proved inductively by judicious use of ceiling signs. However, from now on, we will start ignoring such additive $\pm 1$ terms. Of course they need to be carefully dealt with to obtain a correct proof, but they will always contribute a negligible error, which we will ignore. 
	We will add the symbol \hyperref[foot:danger]{(\textdagger)} to mark the places where we omit such additive errors.} $\frac 12 \ab X$, thus preserving the validity of \eqref{eq:X size ES}. Similarly, in a blue step, we increase $\ab B$ by $1$ and decrease $\ab X$ to at least $\frac 12 \ab X$, again preserving \eqref{eq:X size ES}. By induction, we conclude that \eqref{eq:X size ES} also holds at the end of the process.

	At the end of the process, we thus have
	\[
		N \leq 2^{\ab A + \ab B} \ab X < 2^{k+k}\cdot 1 = 4^k,
	\]
	where we plug in our upper bounds $\ab A<k, \ab B <k, \ab X \leq 1$. This contradiction completes the proof. 
\end{proof}
It is worth noting that, as presented, this argument only proves \cref{thm:ES}---that is, the bound $r(k) \leq 4^k$---rather than the sharper estimate given in \cref{thm:ES binomial}. It is an interesting and instructive exercise to figure out how to modify \cref{alg:ES} to obtain the stronger bound proved in \cref{thm:ES binomial} via a similar ``algorithmic'' proof.

For future reference (and as a hint to solving the exercise above), it is good to observe that the off-diagonal Erd\H os--Szekeres bound $r(k,\ell) \leq \binom{k+\ell}\ell$ can also be obtained in this way. To do so, set $\gamma = \frac{\ell}{k+\ell}$. Then we can modify the Erd\H os--Szekeres algorithm as follows.

\begin{algorithm}{Off-diagonal Erd\H os--Szekeres algorithm}{off diag ES}
	\fixind
	\begin{enumerate}
		
		\item If $\ab X \leq 1$, $\ab A\geq k$, or $\ab B \geq \ell$, stop the process. \label{it:check sizes OD}
		\item Pick a vertex $v \in X$, and check whether $v$ has at least $(1-\gamma)\ab X$ red neighbors in $X$. 
		
		\item  If yes, move $v$ to $A$ and shrink $X$ to the red neighborhood of $v$. That is, update $A \to A\cup \{v\}$ and $X \to X \cap N_R(v)$, and keep $B$ the same. Call this a \emph{red step}.
		\item If not, then $v$ has at least\dangerfoot{}
		$\gamma \ab X$ blue neighbors in $X$. We now move $v$ to $B$, and shrink $X$ to the blue neighborhood of $v$. That is, we update $B \to B\cup \{v\}$ and $X \to X \cap N_B(v)$, and keep $A$ the same. Call this a \emph{blue step}.
		\item Return to step \ref{it:check sizes OD}.
	\end{enumerate}
\end{algorithm} 
The point now is that we obtain the red $K_k$ or blue $K_\ell$ if $\ab A \geq k$ or $\ab B \geq \ell$, and thus we may assume that we do fewer than $k$ red steps and fewer than $\ell$ blue steps. $X$ shrinks by a factor of $1-\gamma$ at every red step, and by a factor of $\gamma$ at every blue step, so throughout the process we have
\[
	\ab X \geq (1-\gamma)^{\ab A} \gamma^{\ab B} N.
\]
On the other hand, the process only terminates if $\ab X \leq 1$, so this implies $N < (1-\gamma)^{-k} \gamma^{-\ell}$. One can check, by Stirling's approximation, that 
\[
	\binom{k+\ell}{\ell} = 2^{o(k)} \gamma^{-\ell} (1-\gamma)^{-k}
\]
for all $\ell \leq k$, and hence this gives a contradiction if we choose $N$  of the form $2^{o(k)}\binom{k+\ell}{\ell}$. This recovers \cref{thm:ES binomial} up to the subexponential error term.

\section{The book algorithm}\label{sec:book alg}
We are now ready to describe the book algorithm of \textcite{2303.09521}. As before, we fix a two-coloring of $E(K_N)$, and assume that there is no monochromatic $K_k$; our goal is to obtain a contradiction if $N$ is sufficiently large.
Throughout the process, we maintain four disjoint sets $A,B,X,Y$, with the following properties: $(A,X)$ is a red book, $(B,X)$ is a blue book, and $(A,Y)$ is another red book\footnote{Equivalently, we could say that $(B,X)$ is a blue book and $(A,X\cup Y)$ is a red book.}. Thus, the only difference from the Erd\H os--Szekeres algorithm is the presence of the new set $Y$. At the end of the process, our goal is to output the pair $(A,Y)$, and to prove that $t \coloneqq \ab A$ and $m \coloneqq \ab Y$ satisfy $m \geq r(k-t,k)$, so that we can apply \cref{lem:book to clique} to obtain a contradiction. We initialize the process with $A = B = \varnothing$, and $X \sqcup Y$ an arbitrary partition of $V(K_N)$ with\dangerfoot{} $\ab X = \ab Y$. By permuting the colors if necessary, we may assume that at the beginning of the process, at least half the edges between $X$ and $Y$ are red.

\begin{center}
	\begin{tikzpicture}
		\draw[fill=red, very thick] (0,0) circle[radius=.5] node[] {$A$};
		\draw[fill=blue!60!ProcessBlue, very thick] (0,-4) circle[radius=.5] node[] {$B$};
		\draw[very thick, fill=white] (-2,-2) ellipse[x radius=1, y radius=.5] node {$X$};
		\scoped[on background layer] \fill[red] (-3,-2) -- (0,0) -- (-1.2,-2) -- cycle;
		\scoped[on background layer] \fill[blue!60!ProcessBlue] (-3,-2) -- (0,-4) -- (-1.2,-2) -- cycle;
		\draw[very thick, fill=white] (2,-2) ellipse[x radius=1, y radius=.5] node {$Y$};
		\scoped[on background layer] \fill[red] (3,-2) -- (0,0) -- (1.2,-2) -- cycle;
	\end{tikzpicture}
\end{center}
As in the Erd\H os--Szekeres algorithm, we will iteratively build this picture by moving vertices from $X$ to $A$ or $B$, and then shrinking $X$ and $Y$. A move from $X$ to $A$ will be called a red step, and a move from $X$ to $B$ will be called a blue step.

What is the advantage of maintaining such a picture? Recall that in the Erd\H os--Szekeres algorithm, $\ab X$ shrinks by a factor of two whenever we do a red or a blue step, hence we end up with $\ab X \geq 2^{-\ab A - \ab B}N$ as in \eqref{eq:X size ES}, yielding the bound $r(k) < 4^k$. However, it is reasonable to hope that since we are imposing ``half as many constraints'' on $Y$ as on $X$---that is, we are only maintaining that the edges between $A$ and $Y$ are red, and not that any edges incident to $Y$ are blue---we may be able to obtain better control on $\ab Y$. Indeed, we might hope that every blue step does not shrink $Y$ at all, while every red step shrinks $Y$ by only a factor of two, as before, yielding\footnote{If we could really obtain such strong control on $\ab Y$, we would show that $r(k) \lesssim 2^k$, a dramatic improvement over \cref{thm:ES}. Unfortunately, and unsurprisingly, the devil is in the details, and a lot of work is needed to make such an approach work, and the extra complications yield a substantially weaker bound.} a bound of $\ab Y \gtrsim 2^{-\ab A}N$. 

In other words, our goal will be to ``sacrifice'' the vertices in $X$, and use them as the fuel we use to build the large red book $(A,Y)$.
This approach comes with a fundamental asymmetry between the colors, in marked contrast to the Erd\H os--Szekeres proof. We will really insist on finding a \emph{red} book $(A,Y)$, and will do our best to build it. Only when doing so is really impossible will we take blue steps.

Because of this, our preferred move would be taking a red step. That is, we would like to pick a vertex $v \in X$, move $v$ to $A$, and update $X \to X \cap N_R(v)$. Moreover, since we need to maintain that $(A,Y)$ is a red book, we will also need to update $Y \to Y \cap N_R(v)$. In particular, when deciding whether to add a vertex $v \in X$ to $A$, we need to check not only that $v$ has many red neighbors in $X$---so that $X$ doesn't shrink too much---but also that $v$ has many red neighbors in $Y$, so that $Y$ doesn't shrink too much. In particular, we see that in addition to tracking the sizes of $A,B,X$, and $Y$, we will also need to track a fifth parameter, the red edge density between $X$ and $Y$. We denote this density by
\[
	p \coloneqq d_R(X,Y) = \frac{e_R(X,Y)}{\ab X \ab Y},
\]
where $e_R(X,Y)$ denotes the number of red edges with one endpoint in $X$ and the other in $Y$.
Recall that, by assumption, we have $p \geq \frac 12$ at the beginning of the process.
Note that every time we add a vertex to $A$ or to $B$ (and thus have to shrink $X$ and potentially $Y$), this red density $p$ might change.
For our simplified exposition of the proof of \cref{thm:3.9999}, we will make the following (completely unjustified) assumption.
\begin{assumption}\label{ass:degree regular}
	At every step of the process, every vertex in $X$ has exactly $p \ab Y$ red neighbors in $Y$, and every vertex in $Y$ has exactly $p \ab X$ red neighbors in $X$.
	In other words, the bipartite graph of red edges between $X$ and $Y$ is bi-regular.
\end{assumption}
We stress again that $X,Y,$ and $p$ change throughout the process, but \cref{ass:degree regular} asserts that whenever such a change happens, we magically end up back with the same bi-regularity.

While \cref{ass:degree regular} is clearly a bogus assumption, it is actually possible to (essentially) make it rigorous. Indeed, the definition of $p$ implies that the vertices in $X$ have, on average, $p \ab Y$ red neighbors in $Y$. A basic but important observation, used frequently in extremal combinatorics, is that one can often convert such average degree conditions to minimum or maximum degree conditions, by deleting a few ``outlier'' vertices. In the rigorous proof of \cref{thm:3.9999}, one must repeatedly ``clean'' $X$ by removing such outliers, and thus one can indeed maintain an approximate version of \cref{ass:degree regular}, at least ensuring that all vertices in $X$ have roughly the same red degree\footnote{It is \emph{much} harder to ensure degree-regularity in both $X$ and $Y$ simultaneously. Luckily, it turns out that degree-regularity in $Y$ is substantially less important in the argument, and in the formal proof one doesn't even ensure an approximate version of it. In its place, one uses a judicious choice of the vertex $v$.}. However, for our exposition, we ignore these important technicalities, and stick with \cref{ass:degree regular}.

\subsection{The steps of the book algorithm}
The two basic steps in the book algorithm will again be red steps and blue steps, as in the Erd\H os--Szekeres algorithm. Note that when we perform a blue step (moving $v \in X$ to~$B$ and updating $X \to X \cap N_B(v)$), we do not need to update~$Y$ at all, since these changes do not affect the fact that $(A,Y)$ is a red book. In particular, thanks to \cref{ass:degree regular}, the red density between~$X$ and~$Y$ remains unchanged during a blue step, since all the remaining vertices in~$X$ still have exactly $p \ab Y$ red neighbors in~$Y$. However, as discussed above, red steps \emph{can} affect $p$, since in a red step we update $X \to X \cap N_R(v)$ and $Y \to Y \cap N_R(v)$, and thus our value of $p$ is updated to
\[
	p' \coloneqq d_R(X \cap N_R(v), Y \cap N_R(v)).
\]
Let us call a vertex \emph{prosperous} if $p' \geq p-\alpha$, for some parameter $\alpha$ we will shortly choose. We will then perform a red step only if there is a vertex $v \in X$ which is prosperous, and which has at least $\frac 12\ab X$ red neighbors in $X$. In such a step, we increase $\ab A$ by $1$, decrease $\ab X$ by a factor of $2$, decrease $Y$ by a factor of $p$ (since $v$ has $p \ab Y$ red neighbors in $Y$, by \cref{ass:degree regular}), and update $p$ to at least $p-\alpha$.

In \cref{alg:ES}, we were always able to do either a red or a blue step, since every vertex in $X$ has at least $\frac 12\ab X$ neighbors in $X$ in one of the colors\dangerfoot. However, if we require that our red vertex $v$ be prosperous, then we may be in a position where neither a red nor a blue step is possible. Namely, we get stuck if all vertices in $X$ have at least $\frac 12 \ab X$ red neighbors in $X$, but none of them is prosperous.

In this case, we implement a \emph{density-boost step}, which is one of the other main innovations of \textcite{2303.09521}. Pick a vertex $v \in X$, and consider the following picture.
\begin{center}
	\begin{tikzpicture}
		\draw[thick] (0,0) ellipse[x radius=1.5, y radius=3];
		\draw[thick] (6,0) ellipse[x radius=1.5, y radius=3];
		\draw[red, thick, fill=white] (6.2,-.3) node[black] (Y) {$U$} ellipse[x radius=1, y radius=1.3] ;
		\node at (0,-3.5) {$X$};
		\node at (6,-3.5) {$Y$};
		\node[vert, label=above:$v$] (v) at (-1.2,0) {};
		\draw[red, thick, fill=white] (.3,-1.3) node[black] (R) {$T$} ellipse[x radius=.8, y radius=1.1] ;
		\draw[blue, thick, fill=white] (.3,1.3) node[black] (B) {$S$} ellipse[x radius=.8, y radius=1.1] ;
		\begin{scope}[on background layer]
			\fill[red] ($(R)+(0,.5)$) -- (-1.2,0) -- ($(R)+(0,-1)$) -- cycle;
			\fill[blue] ($(B)+(0,-.5)$) -- (-1.2,0) -- ($(B)+(0,1)$) -- cycle;
			\draw[dashed, red, ultra thick] ($(R)+(0,.8)$) -- ($(Y)+(0,1)$);
			\draw[dashed, red, ultra thick] ($(R)+(0,-.8)$) -- ($(Y)+(0,-1)$);
		\end{scope}
		\node[rotate=10, align=center] at (3,-.8) {red density\\$< p-\alpha$};
	\end{tikzpicture}
\end{center}
Since $v$ is not prosperous, the red edge density between $T\coloneqq N_R(v) \cap X$ and $U \coloneqq N_R(v) \cap Y$ must be less than $p-\alpha$. However, by \cref{ass:degree regular}, every vertex in $U$ has $p\ab X$ red neighbors in $X$. Therefore, setting $S \coloneqq N_B(v) \cap X$, we find that\dangerfoot
\[
	p \ab X \ab U = e_R(X,U) = e_R(T,U) + e_R(S,U) < (p-\alpha) \ab T \ab U + e_R(S,U).
\]
Rearranging, we conclude that
\[
	e_R(S,U) > \ab U \left( p \ab X - (p-\alpha) \ab T \right).
\]
Let $\beta \coloneqq \frac{\ab S}{\ab X}$, so that $\beta$ records what fraction of the edges from $v$ to the rest of $X$ are blue. Then $\ab S = \beta \ab X$ and\dangerfoot{} $\ab T = (1-\beta) \ab X$, and the above can be rewritten as
\[
	e_R(S,U) > \ab U \ab S \left( \frac p\beta - \frac{(p-\alpha)(1-\beta)}{\beta} \right) = \ab S \ab U \left( p+ \alpha \frac{1-\beta}{\beta} \right),
\]
which implies
\begin{equation}\label{eq:density-boost density}
	d_R(S,U) > p + \alpha \frac{1-\beta}{\beta}.
\end{equation}
Note too that since we cannot do a blue step, we must have $\beta \leq \frac 12$, implying that $d_R(S,U)>p+\alpha$.
In other words, in the bad situation where we cannot perform a red or a blue step, we \emph{can} perform a \emph{density-boost step}, where we replace $X$ by $S=N_B(v) \cap X$, replace $Y$ by $U=N_R(v) \cap Y$, and thus boost the density from $p$ to at least $p+\alpha \frac{1-\beta}{\beta} \geq p+\alpha$.

Note that density-boost steps are expensive, in that they shrink $X$ and $Y$, but don't actually make progress by increasing $\ab A$ or $\ab B$. In particular, we don't \emph{a priori} have any control on how many density-boost steps we perform. Luckily, there is a simple fix to this problem: since we are in any case updating $X \to X \cap N_B(v)$ in a density-boost step, we may add $v$ to $B$ for free, while maintaining the property that $(B,X)$ is a blue book. That is, a density-boost step can also be made a type of blue step, and thus we necessarily perform at most $k$ density-boost steps without creating a blue $K_k$.

The final piece we need before formally defining the book algorithm is to choose $\alpha$, which determines the threshold above which a vertex is considered prosperous. Note that every red step may decrease $p$ by $\alpha$, so if we end up doing up to $k$ red steps, we may decrease $p$ from its initial value of $\frac 12$ to $\frac 12 - \alpha k$. Moreover, whenever we do a red step, we also shrink $Y$ by a factor of (the current value of) $p$. In particular, if $p$ ever drops below (say) $\frac 14$, we are in big trouble: then $Y$ shrinks by a factor of $4$ at every step, and we have no real hope of proving a bound stronger than $r(k) \leq 4^k$. As such, we want to pick $\alpha \leq \frac{\varepsilon}{k}$, so that even after doing $k$ red steps, we have not meaningfully decreased $p$ below its initial value. Here, one can think of $\varepsilon$ as a tiny absolute constant, although in the final analysis we will actually pick $\varepsilon$ to tend to $0$ slowly with $k$.

Unfortunately, there is a trade-off. Recall that we have very little control over the effect of the density-boost steps, because these are the steps we do as a last resort. In fact, essentially our only way of bounding their total effect is the observation that $p \leq 1$ throughout the entire process, which should imply that we cannot do too many density-boost steps, as that would drive the red density up too high. The problem is that a density-boost step only {increases} $p$ by roughly $\alpha$, so if we pick $\alpha\leq \frac{\varepsilon}{k}$, then even if we do $k$ density-boost steps (the maximum possible number), we will only increase the density by $\varepsilon$. In particular, we have no hope of reaching the threshold of $p=1$, where we finally gain some control over the density-boost steps.

The way to resolve this apparent contradiction is to pick $\alpha$ adaptively. Indeed, suppose that at some point in the process, we have reached a red density of, say, $p =0.51$. At this point, it doesn't make sense to have the cutoff be $\alpha=\frac{\varepsilon}{k}$; we wouldn't even mind losing an absolute constant of $1/100$ in the density, since that will only bring us back to our original value of $p$! So we will instead pick $\alpha$ to be dependent on our current value of $p$; namely, we set
\begin{equation}\label{eq:alpha def}
	\alpha(p) \coloneqq
	\begin{cases}
		\frac{\varepsilon}{k} & \text{if } p \leq \frac 12+\frac 1 k,\\
		\varepsilon(p-\frac 12)&\text{otherwise.}
	\end{cases}
\end{equation}
Again, the point of this is that, if we are at some step of the process where $p>\frac 12$, then we can afford to lose more in the density without every dropping $p$ into the ``danger zone'' of being substantially smaller than $\frac 12$. The advantage of this is that the amount we \emph{win} in a density-boost step is itself proportional to $\alpha=\alpha(p)$. So if we have already done some number of density-boost steps, such that $p>\frac 12$, each subsequent density-boost boosts the density even further, at an exponential rate, thus rapidly bringing us closer to the threshold $p=1$.

\subsection{Formal definition of the book algorithm}
With all of these preliminaries, we are finally able to define the book algorithm.
\begin{algorithm}{Book algorithm}{book v1}
	\fixind
\begin{enumerate}
	\item If $\ab X \leq 1$, $\ab A\geq k$, or $\ab B \geq k$, stop the process. \label{it:check sizes book}
	\item Let $p=d_R(X,Y)$ be the current red density between $X$ and $Y$. Define $\alpha=\alpha(p)$ as in \eqref{eq:alpha def}, where $\varepsilon$ is some fixed parameter throughout the process.
	
	\item Check whether some vertex $v \in X$ has at least $\frac 12 \ab X$ blue neighbors in $X$. If yes, perform a \emph{blue step}, by updating
	\[
		A \to A,\qquad B \to B \cup \{v\}, \qquad X \to X\cap N_B(v), \qquad Y \to Y,
	\]
	and return to step \ref{it:check sizes book}.\label{it:blue step book}
	
	\item Check whether some vertex $v \in X$ is {prosperous}, meaning that $d_R(N_R(v) \cap X, N_R(v) \cap Y)\geq p-\alpha$. If yes, perform a \emph{red step}, by updating
	\[
		A \to A \cup \{v\},\qquad B \to B, \qquad X \to X\cap N_R(v), \qquad Y \to Y \cap N_R(v),
	\]
	and return to step \ref{it:check sizes book}.

	\item In the remaining case, pick some vertex $v \in X$. It is not prosperous, and has $\beta\ab X$ blue neighbors in $X$, for some $\beta \leq \frac 12$. We now perform a \emph{density-boost step}, by updating
	\[
		A \to A,\qquad B \to B \cup \{v\}, \qquad X \to X\cap N_B(v), \qquad Y \to Y \cap N_R(v),
	\]
	and return to step \ref{it:check sizes book}.
\end{enumerate}
\end{algorithm}

For future reference, the following table records how the key parameters change during the execution of the book algorithm, following the discussion above. 
\begin{table}[hb]
	\begin{center}
		\tabulinesep=1.2mm
		\begin{tabu}{|c||c|c|c|c|c|}
			\hline
			&$\ab A$ & $\ab B$ & $\ab X$ & $\ab Y$ & $p$ \\ \hhline{|=#*{4}{=|}=|}
			blue step & -- & $+1$ & $\times \frac 12$ & -- & -- \\ \hline
			red step & $+1$ & -- & $\times \frac 12$ & $\times p$ & $-\alpha$ \\ \hline
			density-boost step & -- & $+1$ & $\times \beta$ & $\times p$ & $+ \alpha \frac{1-\beta}{\beta}$
			\\ \hline
		\end{tabu}
	\end{center}
	\caption{How the various parameters evolve during \cref{alg:book v1}. Dashes denote quantities that are unchanged. In general, the entries in the table are \emph{lower bounds}, e.g.\ a density-boost step may increase $p$ by more than $\alpha \frac{1-\beta}{\beta}$, and a red step may shrink $X$ to more than half of its previous size.
	}
	\label{table:book v1}
\end{table}

\subsection{Analysis of the book algorithm}
Suppose that, when the book algorithm ends, we have done $t$ red steps, $s$ density-boost steps, and $b$ blue steps. We may assume that $t <k$ and that $s+b<k$, since otherwise we have found a monochromatic $K_k$. We now collect a number of estimates on the various parameters associated with the process.
\begin{lemma}\label{lem:p never drops}
	We have $p \geq \frac 12 - \varepsilon$ throughout the entire process.
\end{lemma}
\begin{proof}
	As discussed above, every blue step keeps $p$ constant (by \cref{ass:degree regular}), every density-boost step can only increase $p$, and every red step decreases $p$ by at most $\alpha(p)$. Additionally, the choice of $\alpha(p)$ shows that $p-\alpha(p)\geq \frac 12$ whenever $p > \frac 12+\frac 1k$, whereas $\alpha(p)=\frac\varepsilon k$ whenever $p\leq \frac 12 +\frac 1k$. Since we do $t\leq k$ red steps, $p$ can never drop below $\frac 12 - t\cdot\frac\varepsilon k\geq \frac 12 -\varepsilon$.
\end{proof}
It will now be convenient to pick $\varepsilon=k^{-1/4}$, although we note that this choice is not particularly important; many functions of $k$ which tend to $0$ neither too slowly or too quickly would work.
\begin{lemma}\label{lem:Y size}
	At the end of the process, we have $\ab Y \geq 2^{-t-s-o(k)}N$.
\end{lemma}
\begin{proof}
	$Y$ is unchanged by every blue step. On the other hand, during each red or density-boost step, we decrease $Y$ by a factor of $p$, by \cref{ass:degree regular}. By \cref{lem:p never drops}, we have that $p \geq \frac 12-\varepsilon$ at every such step, hence
	\[
		\ab Y \geq \left( \frac 12 - \varepsilon \right)^{t+s} \cdot \frac N2 = 2^{-t-s-o(k)}N,
	\]
	where we plug in our choice of $\varepsilon$ and recall that we start the process with\dangerfoot{} $\ab Y=\frac N2$.
\end{proof}
We next turn to bounding $\ab X$ at the end of the process. Just as in the Erd\H os--Szekeres algorithm, the main point of this is to estimate how many steps we do, since we recall that the process terminates when $\ab X \leq 1$.

Recall that at each density-boost step, we shrink $X$ by a factor of $\beta$, where $\beta$ is defined as the fraction $\ab{N_B(v) \cap X}/\ab X$ of blue neighbors of the currently chosen vertex~$v$. Let $\beta_1,\dots,\beta_s$ be the sequence of values of $\beta$ for each of the $s$ blue steps. Let $\beta$ be the harmonic mean of $\beta_1,\dots,\beta_s$, that is, define $\beta$ by
\[
	\frac 1 \beta = \frac 1s \sum_{i=1}^s \frac{1}{\beta_i}.
\]
\begin{lemma}\label{lem:X size}
	At the end of the process, we have
	\[
		\ab X \geq 2^{-t-b-o(k)} \beta^s N.
	\]
\end{lemma}
\begin{proof}
	Every red or blue step shrinks $X$ by at most a factor\dangerfoot{} of $2$, hence the factor of $2^{-t-b}$. On the other hand, the $i$th density-boost step decreases $\ab X$ by a factor of $\beta_i$. The inequality of arithmetic and geometric means implies that
	\[
		\frac 1 \beta = \frac 1 s \sum_{i=1}^s \frac{1}{\beta_i} \geq \left( \prod_{i=1}^s \frac{1}{\beta_i} \right)^{1/s},
	\]
	hence the contribution of the density-boost steps is
	\[
		\prod_{i=1}^s \beta_i \geq \beta^s.
	\]
	Together with the fact that we begin the process with\dangerfoot{} $\ab X=\frac N2$, this yields the claimed bound.
\end{proof}
The final, and perhaps most important, result we need is an estimate on the number of density-boost steps. As discussed above, we can get a good estimate on this quantity because of the ``dynamic'' choice of $\alpha$; this is the content of the next lemma, which is called the \emph{zig-zag lemma} by \textcite{2303.09521}.
\begin{lemma}[Zig-zag lemma]\label{lem:zigzag}
	We have
	\begin{equation*}
		\sum_{i=1}^s \frac{1-\beta_i}{\beta_i}\leq t + o(k).
	\end{equation*}
\end{lemma}
We won't give a full proof of \cref{lem:zigzag}, but the following sketch captures the main ideas.
\begin{proof}[Proof sketch for \cref{lem:zigzag}]
	For the moment, let us assume that we stay in the regime $p \geq \frac 12 + \frac 1k$. It will be more convenient to reparametrize $p$, by defining $q \coloneqq p-\frac 12$. By our choice of $\alpha$ in \eqref{eq:alpha def}, we have that $\alpha(p)=\varepsilon q$.

	Suppose we do one step of the book algorithm, and thus update $p$ to some new value $p'$ (and update $q$ to $q'=p'-\frac 12$). If the step we do is a blue step, then by \cref{ass:degree regular}, the density $p$ does not change, hence $p'=p$ and $q'=q$. If, instead, we do a red step, then $v$ is prosperous, and hence $p' \geq p - \alpha(p)$. This implies that $q' \geq q - \alpha(p) = q-\varepsilon q = (1-\varepsilon)q$. Finally, if this step is the $i$th density-boost step, then by \eqref{eq:density-boost density} we have that
	\[
		p' \geq p + \alpha(p) \frac{1-\beta_i}{\beta_i}
	\]
	and thus
	\[
		q'\geq q + \alpha(p) \frac{1-\beta_i}{\beta_i} = q \left( 1+\varepsilon \frac{1-\beta_i}{\beta_i} \right).
	\]
	Putting this all together, we conclude that at each step of the algorithm, we have
	\begin{equation}\label{eq:q change}
		\frac{q'}q \geq \begin{cases}
			1&\text{when we do a blue step,}\\
			1-\varepsilon&\text{when we do a red step,}\\
			1+\varepsilon \frac{1-\beta_i}{\beta_i} &\text{when we do the $i$th density-boost step.}
		\end{cases}
	\end{equation}
	Let $q_{\text{final}}$ denote the value of~$q$ at the end of the algorithm, and let $q_{\text{initial}}$ be the value of~$q$ at the beginning of the algorithm. 
	Multiplying \eqref{eq:q change} over all steps of the algorithm, we find that
	\[
		\frac{q_{\text{final}}}{q_{\text{initial}}}\geq (1-\varepsilon)^t \prod_{i=1}^s \left( 1+\varepsilon \frac{1-\beta_i}{\beta_i} \right),
	\]
	since we get a contribution of $1-\varepsilon$ from each of the $t$ red steps and a contribution of $1+\varepsilon \frac{1-\beta_i}{\beta_i}$ from the $i$th density-boost step. Combining this inequality with the approximation $1+x \approx e^x$, an approximation that is valid for sufficiently small\footnote{This approximation can be made rigorous, but we're still cheating in this derivation of \eqref{eq:ratio of q}. We have no guarantee that $\varepsilon \frac{1-\beta_{i}}{\beta_{i}}$ is small, since we have no control over $\beta_i$, and thus no guarantee that the approximation is valid. A correct proof of this lemma would need to separate out the contribution from the steps where $\beta_i$ is very small, and thus where such an approximation is not accurate.} $\ab x$, we find that
	\begin{equation}
		\frac{q_{\text{final}}}{q_{\text{initial}}}\gtrsim e^{-\varepsilon t} \exp \left( \varepsilon \sum_{i=1}^s \frac{1-\beta_i}{\beta_i} \right) = \exp \left( \varepsilon \left( -t + \sum_{i=1}^s \frac{1-\beta_i}{\beta_i} \right) \right).
		\label{eq:ratio of q}
	\end{equation}
	We have that $q_{\text{final}}\leq \frac 12$, since $p \leq 1$ throughout the whole process. On the other hand, since we are assuming that $p \geq \frac 12 + \frac 1k$ throughout, we have that $q_{\text{initial}}\geq \frac 1k$. Therefore, $\frac{\strut q_{\text{final}}}{q_{\text{initial}}}\leq \frac k2\leq k$. Plugging this into \eqref{eq:ratio of q} and taking logarithms, we find that
	\[
		\log k \geq \log \left( \frac{q_{\text{final}}}{q_{\text{initial}}} \right) \gtrsim \varepsilon \left( -t + \sum_{i=1}^s \frac{1-\beta_i}{\beta_i} \right),
	\]
	implying that 
	\[
		\sum_{i=1}^s \frac{1-\beta_i}{\beta_i} \lesssim t + \frac{\log k}{\varepsilon} = t+o(k),
	\]
	where we plug in our choice of $\varepsilon = k^{-1/4}$ in the final equality.

	This proof worked under the assumption that we remain throughout in the range $p \geq \frac 12 + \frac 1k$. Let us now work in the complementary regime, where $p < \frac 12 + \frac 1k$ throughout the whole process. In this case, recalling the definition of $\alpha(p)$ from \eqref{eq:alpha def}, we have
	\begin{equation}\label{eq:p change}
		p' - p \geq
		\begin{cases}
			0 & \text{when we do a blue step,}\\
			- \frac \varepsilon k & \text{when we do a red step,}\\
			+\frac \varepsilon k \cdot \frac{1-\beta_i}{\beta_i} & \text{when we do the $i$th density-boost step.}
		\end{cases}
	\end{equation}
	Adding up \eqref{eq:p change} over all steps of the process, we conclude that
	\begin{equation}\label{eq:sum up p change}
		p_{\text{final}}-p_{\text{initial}} \geq - \frac \varepsilon k t+ \frac \varepsilon k \sum_{i=1}^s \frac{1-\beta_i}{\beta_i} =\frac \varepsilon k \left( -t + \sum_{i=1}^s \frac{1-\beta_i}{\beta_i} \right).
	\end{equation}
	Recall that we had $p_{\text{initial}}\geq \frac 12$, and we are now assuming that we remain in the regime $p < \frac 12 + \frac 1k$ throughout, hence in particular $p_{\text{final}}< \frac 12 + \frac 1k$. Therefore $p_{\text{final}}-p_{\text{initial}} \leq \frac 1k$. Plugging this in to \eqref{eq:sum up p change} and rearranging, we conclude that
	\[
		\sum_{i=1}^s \frac{1-\beta_i}{\beta_i} \leq t + \frac 1 \varepsilon = t + o(k),
	\]
	again by our choice of $\varepsilon$.

	We have thus proved the desired inequality in the two extreme cases, namely when $p \geq \frac 12 + \frac 1k$ throughout the process, and when $p < \frac 12 + \frac 1k$ throughout the process. Of course, in reality, we may move between these two regimes multiple times during the execution of \cref{alg:book v1}. However, by breaking the execution of the algorithm into intervals in which we remain in one regime or the other, it is not too difficult to combine the arguments above and conclude that the claimed inequality always holds.
\end{proof}
As an immediate consequence of \cref{lem:zigzag}, we obtain an upper bound on the number $s$ of density-boost steps.
\begin{lemma}\label{lem:beta LB}
	We have
	\[
		s \leq \left( \frac{\beta}{1-\beta} \right)t + o(k).
	\]
	Equivalently,
	\begin{equation}\label{eq:beta LB}
		\beta \geq (1+o(1)) \frac{s}{s+t}.
	\end{equation}
\end{lemma}
\begin{proof}
	By the definition of $\beta$, we have that
	\[
		\frac 1 s\sum_{i=1}^s \frac{1-\beta_i}{\beta_i} = 
		\frac{1-\beta}{\beta}.
	\]
	Plugging this into \cref{lem:zigzag} shows that
	\[
		s = \left( \frac{\beta}{1-\beta} \right)\sum_{i=1}^s \frac{1-\beta_i}{\beta_i} \leq \left( \frac{\beta}{1-\beta} \right)(t+o(k)).
	\]
	Moreover, since each $\beta_i$ is at most $\frac 12$, we find that $\frac \beta{1-\beta} \leq 1$, yielding the first claimed bound. The second bound follows by solving for\footnote{Again, there is some cheating going on here---one can only obtain the claimed estimate if $s$ is not too small as a function of $k$, in order to absorb the error terms. In the formal proof, one has to separate into cases: the bound \eqref{eq:beta LB} is valid if $s$ is not too small, whereas if $s$ is very small one can complete the proof of \cref{thm:3.9999} via a simpler analysis.} $\beta$.
\end{proof}

\subsection{Proof attempt for \cref{thm:3.9999}}
We are now ready to put everything together. Let $C$ be a constant, which we will optimize later, and let $N=2^{(1+C)k}$. Our plan is to show that if $C$ is chosen appropriately, then $r(k) \leq N= 2^{(1+C)k}$. Since our goal is to prove \cref{thm:3.9999}, we thus hope to be able to prove this for some fixed $C<1$, but for the moment, let us leave $C$ as an unspecified constant.

We proceed by contradiction, so let us assume that there is a two-coloring of $E(K_N)$ with no monochromatic $K_k$. We now run \cref{alg:book v1} on this coloring. We certainly obtain the desired contradiction if the algorithm finds a red or blue $K_k$, so let us assume that this does not happen. Therefore, the process only ends when $\ab X\leq 1$, which by \cref{lem:X size} implies that
\[
	N \leq \beta^{-s} 2^{t+b+o(k)} \leq \beta^{-s} 2^{t+(k-s)+o(k)},
\]
where we plug in the bound $b+s \leq k$, arising from the fact that $B$ never becomes a blue $K_k$. We now plug in the lower bound on $\beta$ from \cref{lem:beta LB} to find that
\begin{equation}
	N \leq \left( \frac{t+s}{s} \right)^s 2^{k+t-s+o(k)}.\label{eq:constraint}
\end{equation}
At this point everything is in terms of the parameters $s$ and $t$, which we expect to scale linearly in $k$, so it is more convenient to reparametrize everything in terms of $x \coloneqq \frac tk, y \coloneqq \frac sk$. In terms of these parameters, and recalling that $N=2^{(1+C)k}$, we can rewrite \eqref{eq:constraint} as
\[
	C -o(1)\leq (x-y) + y \log_2 \left( \frac{x+y}{y} \right)\eqqcolon G(x,y).
\]
Recall that our goal is to obtain a contradiction, and the only thing we have not yet specified is the value of $C$ we choose. In particular, if we pick $C$ to be larger than the maximum value of $G(x,y)$ over the square $[0,1]^2$, then we certainly obtain a contradiction. As our goal is to eventually obtain a contradiction with some fixed $C<1$, we would be happy if this maximum value were less than $1$. However, this is not true, as shown on the following contour plot; the maximum value of $G$ is roughly $1.33$.
\begin{center}
	\includegraphics[width=6.5cm]{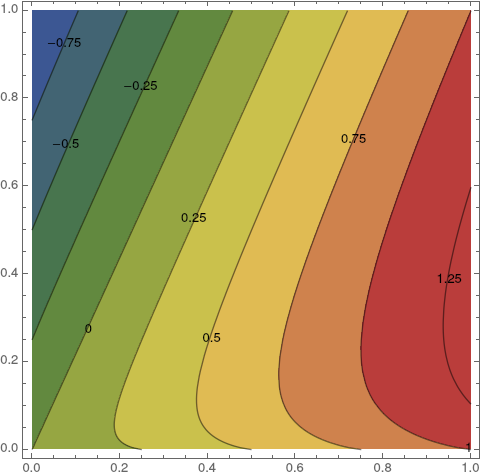}
\end{center}
Of course, if we recall our original strategy, it is way too much to hope for that the maximum of $G$ is less than $1$. Indeed, the whole point of the book algorithm was to output the book $(A,Y)$, and to ensure that its parameters are good enough to apply \cref{lem:book to clique}. 

What are the parameters of this book? Well, we have that $\ab A=t$ by definition, and 
\[
	m \coloneqq \ab Y \geq 2^{-t-s-o(k)}N
\]
by \cref{lem:Y size}. By \cref{lem:book to clique}, we know that if $m \geq r(k-t,k)$, we find a monochromatic $K_k$, yielding our desired contradiction. Thus, we may assume that $m<r(k-t,k)$, implying that
\begin{equation}
	N \leq 2^{t+s+o(k)} m < 2^{t+s+o(k)} r(k-t,k).\label{eq:constraint 2}
\end{equation}
By \cref{thm:ES binomial}, we know that
\[
	r(k-t,k) \leq \binom{2k-t}{k-t}.
\]
A useful upper bound on binomial coefficients is that $\binom ab \leq 2^{a H(b/a)}$, where $H(z) \coloneqq -z\log_2 z - (1-z)\log_2(1-z)$ is the \emph{binary entropy function}. Plugging this in, we find that
\[
	\log_2 r(k-t,k) \leq \log_2 \binom{2k-t}{k-t} \leq (2k-t) H \left( \frac{k-t}{2k-t} \right)=k \left[ (2-x) H \left( \frac{1-x}{2-x} \right) \right].
\]
Taking logarithms of \eqref{eq:constraint 2} and dividing by $k$ shows that
\[
	C-o(1) \leq -1 + (x+y) + (2-x)H \left( \frac{1-x}{2-x} \right) \eqqcolon F(x,y).
\]
Putting all of this together, we have shown that either we derive the claimed contradiction, or $C-o(1)\leq \min\{F(x,y),G(x,y)\}$. Again, we have the freedom to choose $C$, so we can obtain the desired contradiction if we set $C$ to be larger than the maximum of $\min\{F(x,y),G(x,y)\}$ on the square $[0,1]^2$. In particular, as our goal is to pick $C<1$, we are done if $\min\{F(x,y),G(x,y)\}<1$ for all $x,y\in [0,1]$. Here is a contour plot of $F$:
\begin{center}
	\includegraphics[width=6.5cm]{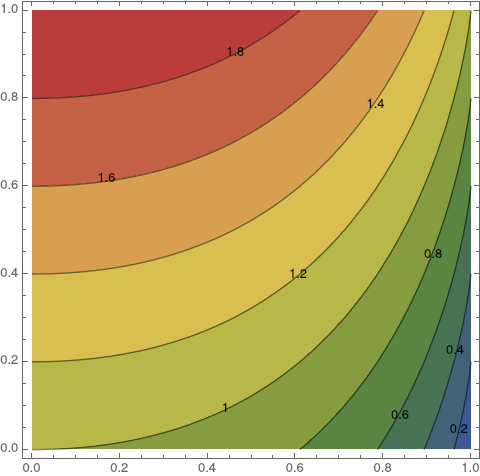}
\end{center}
This looks great! The areas where $F$ is large seem to be different from the areas where $G$ is large, so there should be no problem to show that their maximum is always strictly less than $1$. In fact, here are the regions where $F>1$ and $G>1$.
\begin{center}
	\includegraphics[width=6.5cm]{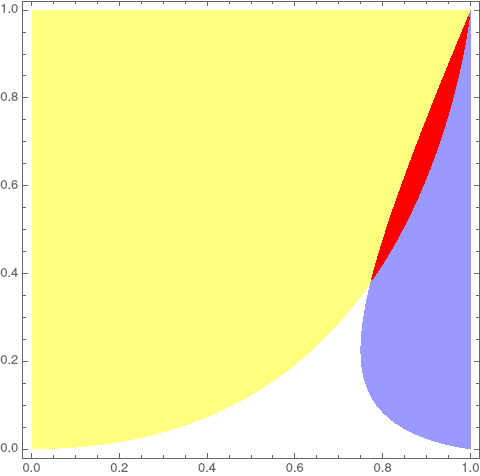}
\end{center}
Bad news! There's a big red region where both functions are greater than $1$, and our whole proof strategy fails. In fact, one can check that $\min\{F(x,y),G(x,y)\}$ attains a maximum value of roughly $1.054$. That is, in order to obtain a contradiction, the smallest value of $C$ we could pick is $1.054$, and thus this whole complex proof is only able to show that $r(k) \leq 2^{2.054k}\approx 4.15^k$, which is worse than the bound of \cref{thm:ES}.

\section{Rescuing the argument}\label{sec:rescue}
The fact that $\min\{F(x,y),G(x,y)\}>1$ for some $(x,y) \in [0,1]^2$ is a fundamental obstruction to this approach. In order to overcome it, we will use two tricks, both of which involve tweaking the book algorithm. 
\subsection{Changing the cutoff}

The first new idea is to examine our criterion for deciding whether to do red or blue steps. Recall that, as in \cref{alg:ES}, we do a blue step if some vertex in $X$ has at least $\frac 12 \ab X$ blue neighbors in $X$, and otherwise we do a red or density-boost step. In the Erd\H os--Szekeres setting, this is the optimal choice: since the argument is symmetric in the two colors, it would be strictly worse to use any other cutoff.

However, the book algorithm is highly asymmetric, so we should re-examine this assumption. Recall that at the end of the process, we output the red book $(A,Y)$, where $\ab A=t$ and $\ab Y \geq 2^{-t-s-o(k)}N$ by \cref{lem:Y size}. The fact that $\ab Y$ decays like $2^{-t}N$ is unavoidable, but the fact that $\ab Y$ decays exponentially in $s$ shows that density-boost steps are very expensive, in terms of making this trade-off very bad. As such, we should try to minimize the number $s$ of density-boost steps we do, in terms of $t$. Since \cref{lem:beta LB} tells us that $s \leq \frac \beta{1-\beta}t+o(k)$, the natural way to decrease $s$ is to decrease $\beta$.

To achieve this, we do the following. We pick a number $\mu \in [0,1]$, which will be fixed throughout the argument. In step \ref{it:blue step book} of \cref{alg:book v1}, we now perform a blue step if some vertex in $X$ has at least $\mu\ab X$ blue neighbors in $X$; otherwise, we proceed to the subsequent steps of the algorithm unchanged. An important effect of this choice is that now, when we perform the $i$th density-boost step, the parameter $\beta_i$ is constrained to be at most $\mu$, and thus also $\beta \leq \mu$ at the end of the process. In particular, if we pick $\mu < \frac 12$, we will have accomplished our goal of decreasing $s$ relative to $t$. This suggests we should pick $\mu$ very small, but of course there is a trade-off---if $\mu$ is very small then every blue step decreases $\ab X$ by a lot, and thus the process will terminate quickly. To balance these two effects, we want to pick $\mu$ to be neither too large nor too small. For completeness, here is the description of our modified book algorithm.
\begin{algorithm}{Book algorithm with cutoff $\mu$}{book v2}
	\fixind
\begin{enumerate}
	\item If $\ab X \leq 1$, $\ab A\geq k$, or $\ab B \geq k$, stop the process. \label{it:check sizes book v2}
	\item Let $p=d_R(X,Y)$ be the current red density between $X$ and $Y$. Define $\alpha=\alpha(p)$ as in \eqref{eq:alpha def}, where $\varepsilon$ is some fixed parameter throughout the process.

	\item Check whether some vertex $v \in X$ has at least $\mu \ab X$ blue neighbors in $X$. If yes, perform a \emph{blue step}, by updating
	\[
		A \to A,\qquad B \to B \cup \{v\}, \qquad X \to X\cap N_B(v), \qquad Y \to Y,
	\]
	and return to step \ref{it:check sizes book v2}.\label{it:blue step book v2}
	
	\item Check whether some vertex $v \in X$ is {prosperous}, meaning that $d_R(N_R(v) \cap X, N_R(v) \cap Y)\geq p-\alpha$. If yes, perform a \emph{red step}, by updating
	\[
		A \to A \cup \{v\},\qquad B \to B, \qquad X \to X\cap N_R(v), \qquad Y \to Y \cap N_R(v),
	\]
	and return to step \ref{it:check sizes book v2}.

	\item In the remaining case, pick some vertex $v \in X$. It is not prosperous, and has $\beta\ab X$ blue neighbors in $X$, for some $\beta \leq \mu$. We now perform a \emph{density-boost step}, by updating
	\[
		A \to A,\qquad B \to B \cup \{v\}, \qquad X \to X\cap N_B(v), \qquad Y \to Y \cap N_R(v),
	\]
	and return to step \ref{it:check sizes book v2}.
\end{enumerate}
\end{algorithm}
\begin{table}[hb!]
	\begin{center}
		\tabulinesep=1.2mm
		\begin{tabu}{|c||c|c|c|c|c|}
			\hline
			&$\ab A$ & $\ab B$ & $\ab X$ & $\ab Y$ & $p$ \\ \hhline{|=#*{4}{=|}=|}
			blue step & -- & $+1$ & $\times \mu$ & -- & -- \\ \hline
			red step & $+1$ & -- & $\times (1-\mu)$ & $\times p$ & $-\alpha$ \\ \hline
			density-boost step & -- & $+1$ & $\times \beta$ & $\times p$ & $+ \alpha \frac{1-\beta}{\beta}$
			\\ \hline
		\end{tabu}
	\end{center}

	\caption{How the various parameters evolve during \cref{alg:book v2}. 
	The only difference from \cref{table:book v1} is that blue and red steps shrink $X$ by factors of $\mu$ and $1-\mu$, respectively.
	}
	\label{table:book v2}
\end{table}

In this modified book algorithm, \cref{lem:zigzag,lem:beta LB,lem:p never drops,lem:Y size} remain true; the only change is that \cref{lem:X size} needs to be modified to the following statement, reflecting the fact that each blue (resp.\ red) step shrinks $X$ by a factor\dangerfoot{} of $\mu$ (resp.\ $1-\mu$) in the worst case. The proof is otherwise identical to that of \cref{lem:X size}.
\begin{lemma}[Modified \cref{lem:X size}]\label{lem:modified X size}
	At the end of the process, we have
	\[
		\ab X \geq 2^{-o(k)} (1-\mu)^t \mu^b \beta^s N.
	\]
	In particular, since $b+s\leq k$, we have
	\[
		\ab X \geq 2^{-o(k)} (1-\mu)^t \mu^{k-s} \beta^s N.
	\]
\end{lemma}
Since the process terminates when $\ab X \leq 1$, we conclude from \cref{lem:modified X size} that
\begin{equation}\label{eq:bad mu exponent}
	N \leq 2^{o(k)} (1-\mu)^{-t} \mu^{-(k-s)} \beta^{-s} \leq 2^{o(k)} (1-\mu)^{-t} \mu^{-(k-s)} \left( \frac{s+t}{s} \right)^s,
\end{equation}
where the final inequality follows from the lower bound on $\beta$ in \cref{lem:beta LB}. Taking logarithms and dividing by $k$, we conclude that
\[
	C-o(1) \leq -1 + x \log_2 \left( \frac{1}{1-\mu} \right) + (1-y) \log_2 \frac 1 \mu + y \log_2 \left( \frac{x+y}{y} \right) \eqqcolon G_\mu(x,y).
\]
Note that in the case $\mu=\frac 12$, we precisely recover the previous function $G$, which of course makes sense as we are then recovering \cref{alg:book v1}. Here are contour plots of $G_\mu$ for $\mu \in \{\frac 1{10},\frac 2{10},\frac 3{10},\frac 4{10}\}$.
\begin{center}
	\includegraphics[width=3.7cm]{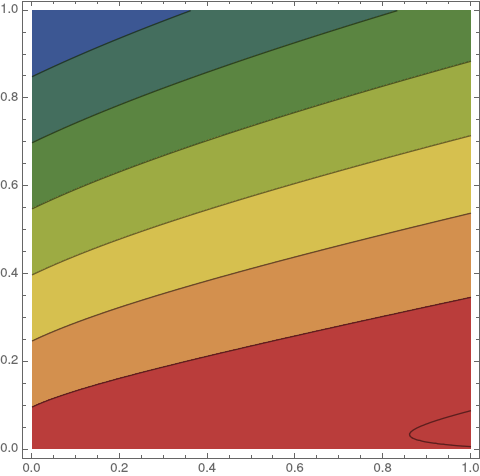}
	\includegraphics[width=3.7cm]{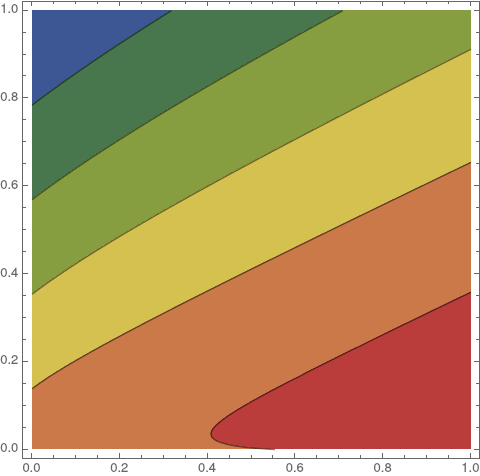}
	\includegraphics[width=3.7cm]{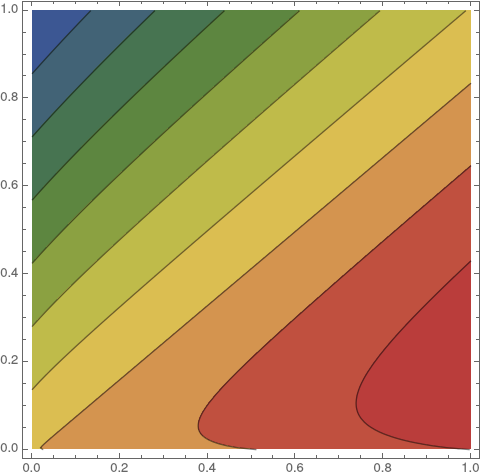}
	\includegraphics[width=3.7cm]{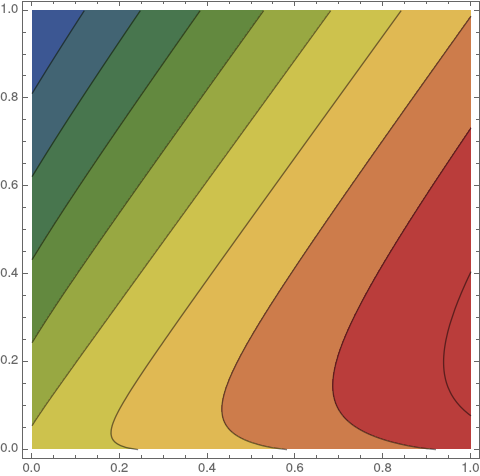}
\end{center}
And here are pictures of the regions where $F>1$ and $G_\mu>1$, for $\mu \in \{\frac 1{10},\frac 2{10},\frac 3{10},\frac 4{10}\}$.
\begin{center}
	\includegraphics[width=3.7cm]{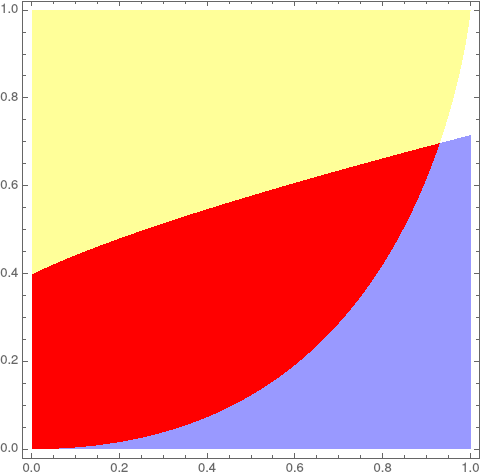}
	\includegraphics[width=3.7cm]{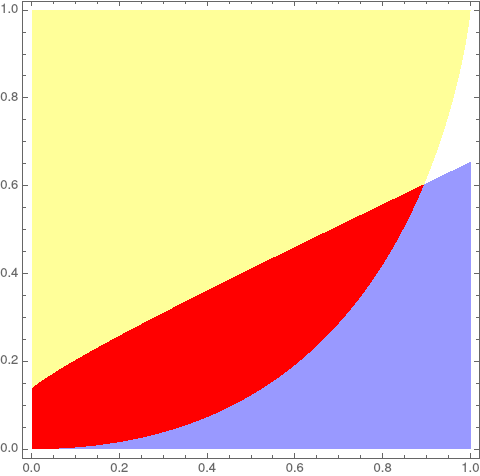}
	\includegraphics[width=3.7cm]{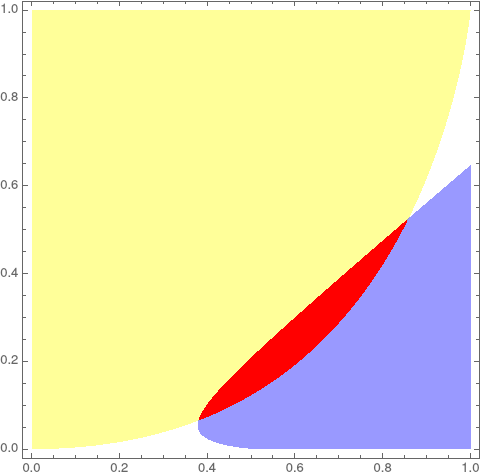}
	\includegraphics[width=3.7cm]{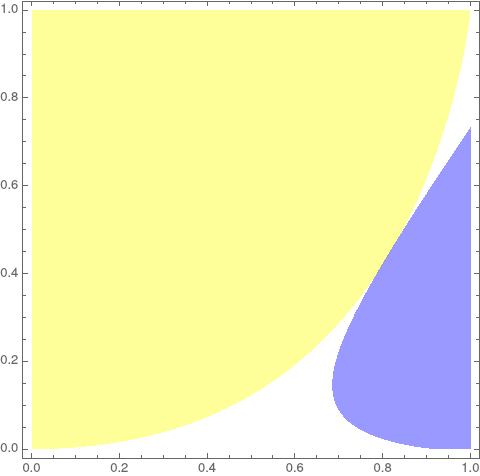}
\end{center}
It looks like we're already done at $\mu=\frac 4{10}=\frac 25$, but unfortunately we're not: one can check that $\min\{F(x,y),G_{\frac 25}(x,y)\}$ attains a maximum value of $1.0017$, hence we only obtain a bound of $r(k) \leq 4.006^k$. Here is a closer view of what happens at $\mu=\frac 25$:
\begin{center}
	\includegraphics[width=8cm]{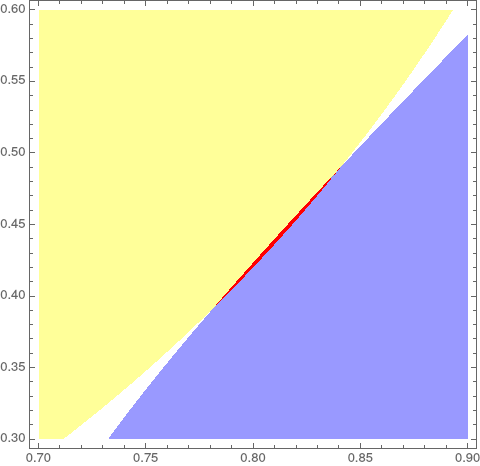}
\end{center}
But we're definitely making progress! The bad red region is extremely small now, and our maximum value of $\min\{F,G_\mu\}$ is extraordinarily close to $1$. Unfortunately, one can check that no choice of $\mu$ will actually decrease this value below $1$---which would complete the proof---so another idea is needed.

\subsection{Off-diagonal Ramsey numbers}
So far, we have played with the parameter~$\mu$ in order to vary the region where $G_\mu>1$, and have almost succeeded in making it disjoint from the region where $F>1$. We will now try to tweak~$F$, in order to move this latter region. Recall that the way we defined~$F$ was in terms of an upper bound on $r(k-t,k)$. If we can obtain a better upper bound on $r(k-t,k)$, then $F$~will decrease, and we may be in business. In fact, we don't need to improve the upper bound on $r(k-t,k)$ in all cases; it suffices to improve this upper bound for pairs $(k-t,k)$ near the problematic region where both~$F$ and $G_{\frac 25}$ are greater than~$1$. Since this problematic region is near $x \approx 0.75$, we could hope to improve the upper bound on $r(k-t,k)$ where $k-t \approx 0.75k$, or equivalently on $r(k,\ell)$ where $\ell \approx \frac k4$.

There is actually a good reason to expect this to work. Recall \cref{alg:off diag ES}, which yields an upper bound on such off-diagonal Ramsey numbers. In that algorithm, we choose whether to do red or blue steps based on the cutoff $\gamma=\frac{\ell}{k+\ell}$. If we just blindly import the same idea into the book algorithm, it makes sense to set $\mu \approx \frac{\ell}{k+\ell}$ in order to upper-bound $r(k,\ell)$. In case $\ell \approx \frac k4$, we have $\mu \approx \frac 15$. In our argument above, we saw that it is good to take $\mu$ small, except for the trade-off that now $X$ shrinks by a factor of $\mu$ for every blue step. However, in this regime, we will do at most $\ell$ blue steps, and $\mu^\ell \approx (\frac 15)^{\frac k4} \approx 0.67^k$; in contrast, in the argument above, the blue steps shrink $X$ by $(\frac 25)^k = 0.4^k$, which is a much more significant decrease. Hence we may expect the trade-offs to work well for us.

For completeness, here is our final book algorithm, suited for upper-bounding $r(k,\ell)$. We set $\mu=\frac{\ell}{k+\ell}$ and $\varepsilon = k^{-1/4}$. We initiate $A=B=\varnothing$, and $X\sqcup Y$ an arbitrary partition of $V(K_N)$ into two equally-sized parts. Let $p_{\text{initial}}=d_R(X,Y)$ be the density of red edges between $X$ and $Y$ at the beginning of the process, and define
\begin{equation}\label{eq:off diag alpha def}
	\alpha(p)\coloneqq
	\begin{cases}
		\frac{\varepsilon}{k} & \text{if } p \leq p_{\text{initial}}+\frac 1 k,\\
		\varepsilon(p-p_{\text{initial}})&\text{otherwise.}
	\end{cases}
\end{equation}
The algorithm is then as follows.
\begin{algorithm}{Off-diagonal book algorithm}{off diag book}
	\fixind
	\begin{enumerate}
		\item If $\ab X \leq 1$, $\ab A\geq k$, or $\ab B \geq \ell$, stop the process. \label{it:check sizes book OD}
		\item Let $p=d_R(X,Y)$ be the current red density between $X$ and $Y$. Define $\alpha=\alpha(p)$ as in \eqref{eq:off diag alpha def}.
		
		\item Check whether some vertex $v \in X$ has at least $\mu \ab X$ blue neighbors in $X$. If yes, perform a \emph{blue step}, by updating
		\[
		A \to A,\qquad B \to B \cup \{v\}, \qquad X \to X\cap N_B(v), \qquad Y \to Y,
		\]
		and return to step \ref{it:check sizes book OD}.
		
		\item Check whether some vertex $v \in X$ is {prosperous}, meaning that $d_R(N_R(v) \cap X, N_R(v) \cap Y)\geq p-\alpha$. If yes, perform a \emph{red step}, by updating
		\[
		A \to A \cup \{v\},\qquad B \to B, \qquad X \to X\cap N_R(v), \qquad Y \to Y \cap N_R(v),
		\]
		and return to step \ref{it:check sizes book OD}.
		
		\item In the remaining case, pick some vertex $v \in X$. It is not prosperous, and has $\beta\ab X$ blue neighbors in $X$, for some $\beta \leq \mu$. We now perform a \emph{density-boost step}, by updating
		\[
		A \to A,\qquad B \to B \cup \{v\}, \qquad X \to X\cap N_B(v), \qquad Y \to Y \cap N_R(v),
		\]
		and return to step \ref{it:check sizes book OD}.
	\end{enumerate}
\end{algorithm} 
Apart from the choice of $\mu = \frac{\ell}{k+\ell}$, this algorithm is identical to \cref{alg:book v2}, except that we now stop when $\ab B \geq \ell$, rather than when $\ab B \geq k$ as before. In particular, \cref{table:book v2} still gives the relevant changes in the parameters.
Unfortunately, there is an additional complication introduced by moving to the off-diagonal setting. Before, when we sought to upper-bound $r(k)$, we could assume that the initial red density $p_{\text{initial}}$ was at least $\frac 12$, by simply swapping the roles of the two colors if necessary. However, once we are in the off-diagonal setting, this is no longer allowed, and we may have no control on $p_{\text{initial}}$. Let us make another completely unjustified assumption.
\begin{assumption}\label{ass:p_0}
	At the beginning of the process, we have $p_{\mathrm{initial}} \geq \frac{k}{k+\ell}=1-\mu$.
\end{assumption}
Note that this is a natural assumption, since \cref{alg:off diag ES} ``predicts'' that $\frac{k}{k+\ell}$ is roughly the correct red density to expect, in the sense that in the analysis of \cref{alg:off diag ES}, this red density is the worst-case occurrence. That is, if \cref{ass:p_0} is false ``robustly'', then \cref{alg:off diag ES} should already prove a stronger bound than $r(k,\ell) \leq \binom{k+\ell}{\ell}$.  
In fact, one can essentially force \cref{ass:p_0} to hold because of such an argument; if we start with $p_{\text{initial}}<\frac{k}{k+\ell}$, we can run a number of steps of the Erd\H os--Szekeres algorithm, until we end up with $p \geq \frac{k}{k+\ell}$. If this never happens, then \cref{alg:off diag ES} itself will prove that $r(k,\ell) \ll \binom{k+\ell}{\ell}$.

Given \cref{ass:p_0}, we can conclude that \cref{lem:modified X size,lem:zigzag,lem:beta LB} remain true for \cref{alg:off diag book}. Moreover, we can prove the following modified version of \cref{lem:p never drops,lem:Y size} (combined into a single statement), whose proof is essentially unchanged. 
\begin{lemma}[Modified \cref{lem:p never drops,lem:Y size}]\label{lem:off diagonal Y size}
	We have $p \geq p_{\mathrm{initial}} - \varepsilon \geq (1-\mu)-\varepsilon$ throughout the entire process. Therefore, at the end of the process, we have
	\[
		\ab Y \geq (1-\mu)^{t+s+o(k)}N.
	\]
\end{lemma}

With all of this setup, we are finally able to prove\footnote{One should really write ``prove'', since everything here is dependent on the unjustified \cref{ass:degree regular,ass:p_0}, as well as on the key \cref{lem:zigzag}, which we did not rigorously prove.
Additionally, the bound in \cref{thm:off diagonal exp} is stronger than any result proved by \textcite{2303.09521}, and this too is a consequence of the fact that we are being unrigorous, especially with the verification of certain numerical inequalities. However, \cref{thm:off diagonal exp} is true; a rigorous proof of a stronger statement is given by \textcite[Corollary 6]{2407.19026}.
} an exponentially-improved upper bound on $r(k,\ell)$. 
\begin{theorem}\label{thm:off diagonal exp}
	We have $r(k,\ell) \leq 2^{-\frac 29 \ell +o(k)} \binom{k+\ell}{\ell}$ for all $\ell \leq \frac k4$.
\end{theorem}
Note that in this theorem, the gain over \cref{thm:ES binomial} is exponential in $\ell$, and not in~$k$. This is natural, and the best we could hope for. Indeed, if $\ell=o(k)$, then the bound in \cref{thm:ES binomial} is already subexponential in~$k$, so it is impossible to improve it by a factor of $2^{-\delta k + o(k)}$ for any fixed $\delta>0$.
\begin{proof}[Proof of \cref{thm:off diagonal exp}]
	Let $C$ be a constant that we will optimize later, and let $N=2^{(1+C)k}$. We fix a two-coloring of $E(K_N)$, and assume for contradiction that there is no red $K_k$ or blue $K_\ell$ in this coloring.
	We apply \cref{alg:off diag book}, with $\mu = \frac{\ell}{k+\ell} \leq \frac 15$. 
	Note that this choice of $\mu$ implies that $\frac\ell k= \frac \mu{1-\mu}$.
	If we output that $A$ is a red $K_k$ or $B$ is a blue $K_\ell$, then we have obtained a contradiction, hence we can assume this does not happen. Therefore, the process only terminates when $\ab X \leq 1$, and we also have that $b+s\leq \ell$. Plugging this into \cref{lem:modified X size}, we find that
	\begin{equation*}
		N \leq 2^{o(k)} (1-\mu)^{-t} \mu^{-(\ell-s)} \beta^{-s} \leq 2^{o(k)} (1-\mu)^{-t} \mu^{-(\ell-s)} \left( \frac{s+t}{s} \right)^s.
	\end{equation*}
	Note that we obtain a better exponent on $\mu$ than we had in \eqref{eq:bad mu exponent}, because the assumption $b+s\leq \ell$ is stronger than what we had before; this is precisely the extra strength gained by moving to the off-diagonal setting. Taking logarithms and dividing by $k$ shows that
	\[
		C-o(1) \leq -1 + x \log_2 \left( \frac{1}{1-\mu} \right) + \left(\frac\mu{1-\mu}-y\right) \log_2 \frac 1 \mu + y \log_2 \left( \frac{x+y}{y} \right) \eqqcolon \wt G_\mu(x,y),
	\]
	where the only difference between $G_\mu$ and $\wt G_\mu$ is the the term $\frac{\mu}{1-\mu}$ in the latter, which is simply $1$ in the former. It comes from the $\ell$ in the exponent; upon dividing by $k$ we obtain $\frac\ell k=\frac \mu{1-\mu}$.

	Additionally, by \cref{lem:off diagonal Y size}, we have
	\[
		\ab Y \geq (1-\mu)^{t+s+o(k)} N.
	\]
	If $\ab Y \geq r(k-t,\ell)$, then we obtain a contradiction by \cref{lem:book to clique}, so we may assume that $\ab Y < r(k-t,\ell)$. Taking logarithms and dividing by $k$ again shows that
	\begin{equation}\label{eq:Y cons off diag}
		C-o(1) \leq -1+ (x+y) \log_2 \left( \frac{1}{1-\mu} \right)+\frac 1k \log_2 r(k-t,\ell).
	\end{equation}
	By \cref{thm:ES}, we have
	\begin{align*}
		\log_2 r(k-t,\ell) &\leq \log_2 \binom{k-t+\ell}{\ell} \\
		&\leq (k-t+\ell) H \left( \frac{\ell}{k-t+\ell} \right)\\
		&= k \cdot \left( 1-x+\frac{\mu }{1-\mu} \right) H \left( \frac{\mu/(1-\mu)}{1-x+\mu/(1-\mu)} \right).
	\end{align*}
	Plugging this into \eqref{eq:Y cons off diag} shows that
	\begin{align*}
		C- o(1) &\leq -1+ (x+y) \log_2 \left( \frac{1}{1-\mu} \right)+\left( 1-x+\frac{\mu }{1-\mu} \right) H \left( \frac{\mu/(1-\mu)}{1-x+\mu/(1-\mu)} \right) \\
		&\eqqcolon \wt F_\mu(x,y).
	\end{align*}
	We are no longer trying to beat the bound $r(k)\leq 4^k$, so our goal is no longer to obtain a contradiction for some $C<1$. Instead, we are comparing to $\frac 1k \log_2 \binom{k+\ell}{\ell}$, which equals $(1+\frac{\mu }{1-\mu}) H(\mu)+o(1)$, and thus our goal is to obtain a contradiction for some fixed $C<(1+\frac{\mu }{1-\mu}) H(\mu) -1$. That is, we would like to show is that for all $\mu \leq \frac 15$, we have $\min \{\wt F_\mu(x,y), \wt G_\mu(x,y)\} < (1+\frac{\mu}{1-\mu})H(\mu)-1$ for all $x,y \in [0,1]$. In fact, we hope to prove this inequality with some slack, so that we gain an improvement in the exponent. 

	Recall that our goal is to prove a gain over \cref{thm:ES binomial} that is exponential in $\ell$. As such, the slack we get in this inequality should scale like $\frac\ell k=\frac{\mu}{1-\mu}$. That is, we would like to prove an inequality of the form
	\[
		\min \{\wt F_\mu(x,y), \wt G_\mu(x,y)\} < \left(1+\frac{\mu}{1-\mu}\right)H(\mu)-1 - \delta \frac{\mu}{1-\mu},
	\]
	where $\delta>0$ is some absolute constant; such a bound would prove that $r(k,\ell) \leq 2^{-\delta \ell + o(k)} \binom{k+\ell}\ell$. 

	Such an inequality holds! In fact, one can check that for $\mu \leq \frac 15$, we may take $\delta$ as large as $\frac 29$. Indeed, here is a plot of the regions where $\wt F_\mu > (1+\frac{\mu}{1-\mu})H(\mu)-1-\frac 29 \frac{\mu}{1-\mu}$ and $\wt G_\mu > (1+\frac{\mu}{1-\mu})H(\mu)-1-\frac 29 \frac{\mu}{1-\mu}$, respectively, for $\mu=\frac 15$. One can verify that the regions only move further apart as $\mu$ decreases, so $\mu=\frac 15$ is the worst case.
	\begin{center}
		\includegraphics[width=8cm]{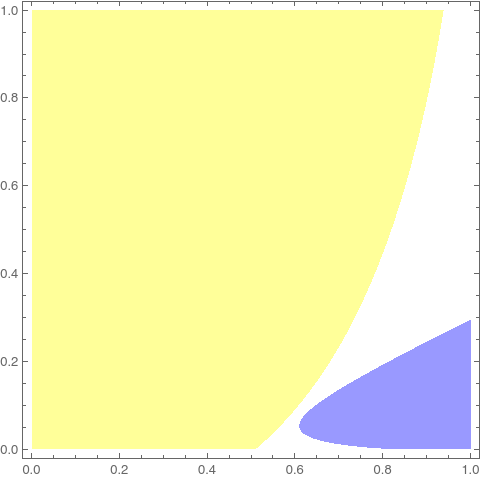}
	\end{center}
	This shows that we do indeed get a contradiction if we set $C = (1+\frac{\mu}{1-\mu})H(\mu)-\frac 29 \frac{\mu}{1-\mu}$, proving the bound 
	\[
		r(k,\ell) \leq 2^{\left(1+\frac{\mu}{1-\mu}H(\mu)-\frac 29 \frac{\mu}{1-\mu}+o(1)\right)k} = 2^{-\frac 29 \ell+o(k)} \binom{k+\ell}{\ell} 
	\]
	for all $\ell \leq \frac k4$.
\end{proof}

\subsection{Back to the diagonal}
Now that we have an upper bound on $r(k,\ell)$ for $\ell \leq \frac k4$, we can finally complete the proof of \cref{thm:3.9999}. We will actually prove the following bound, which is a little bit stronger. The only reason we obtain a stronger bound than \textcite{2303.09521} is the numerical computations: the authors rigorously justify every numerical inequality, which is often substantially simpler to do if one proves a slightly weaker bound, whereas we content ourselves with ``proving'' the numerical bounds through pictures. As mentioned in \cref{sec:intro}, much stronger bounds were recently proved by \textcite{2407.19026} via an alternative analysis. It follows from their results that \cref{thm:3.96} is true, even though we do not provide a rigorous proof.
\begin{theorem}\label{thm:3.96}
	We have $r(k) \leq 2^{(2-\frac 3{200}+o(1))k} \approx 3.96^k$.
\end{theorem}
\begin{proof}
	Let $C$ be a constant that we will optimize later. Let $N=2^{(1+C)k}$, and fix a two-coloring of $E(K_N)$, which we assume for contradiction has no monochromatic $K_k$. We run \cref{alg:book v2} with $\mu=\frac 25$. Thanks to \cref{thm:off diagonal exp} (plus \cref{thm:ES binomial}), we know that
	\[
		r(k-t,k) \leq 
		\begin{cases}
			\displaystyle \binom{2k-t}{k-t}&\text{if }t < \frac 34 k,\\
			\displaystyle 2^{-\frac 29 (k-t) +o(k)}\binom{2k-t}{k-t} &\text{if } t \geq \frac 34 k.
		\end{cases}
	\]
	Recall that we obtain a contradiction if $\ab Y \geq r(k-t,k)$ at the end of the process, hence we may assume that $\ab Y < r(k-t,k)$. Combining this with \cref{lem:Y size}\footnote{We are back to the diagonal setting, so we may assume that $p_{\text{initial}} \geq \frac 12$. Therefore \cref{lem:p never drops,lem:Y size} are again valid.}, we see that we get a contradiction if
	\begin{align*}
		C-o(1) &\leq -1+(x+y) + \frac 1k \log_2 r(k-t,k)\\
		&\leq \begin{cases}
			-1+(x+y)+(2-x)H(\frac{1-x}{2-x})&\text{if }x < \frac 34,\\
			-1+(x+y)-\frac 29(1-x) + (2-x)H(\frac{1-x}{2-x})&\text{if }x \geq \frac 34
		\end{cases}\\
		&=F(x,y) - \frac 29(1-x) \bm 1_{x\geq \frac 34}\\
		&\eqqcolon \wh F(x,y),
	\end{align*}
	where $\bm 1_{x\geq \frac 34}$ denotes the indicator function for the event $x \geq \frac 34$.
	In particular, it suffices for us to prove that $\min \{\wh F(x,y), G_{\frac 25}(x,y)\} \leq 1-\delta$ for all $x,y \in [0,1]$, where $\delta>0$ is a constant that will end up in the exponent in $N$. 
	
	This indeed works! Here are the plots of where $\wh F$ and $G_{\frac 25}$ are greater than $1$; the second plot is just zoomed in to show the ``dangerous area'', where the two regions no longer intersect.
	\begin{center}
		\includegraphics[height=7cm]{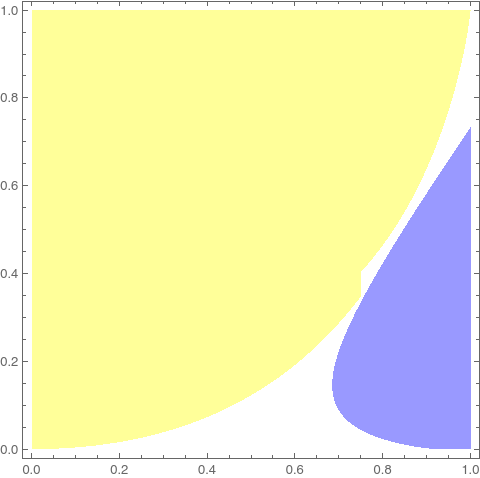}
		\includegraphics[height=7cm]{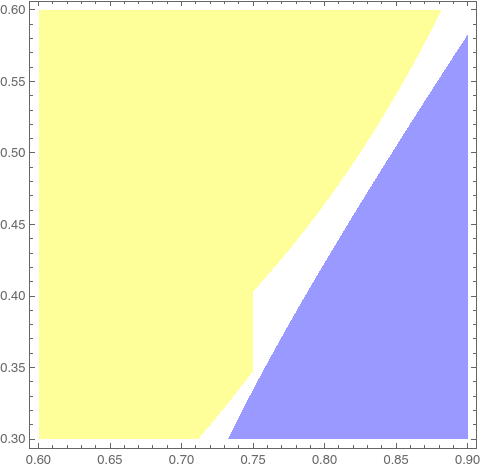}
	\end{center}
	In fact, one can check that $\max_{x,y \in [0,1]} \min \{\wh F(x,y), G_{\frac 25}(x,y)\}<0.985$. Therefore, we obtain a contradiction if we set $C= 0.985 = 1-\frac{3}{200}$, proving that $r(k) \leq 2^{(2-\frac 3{200}+o(1))k}$, as claimed.
\end{proof}

\section{The symmetric book algorithm}\label{sec:symmetric alg}
In this section, we describe the new proof that $r(k) \leq (4-\delta)^k$, due to \textcite{2410.17197}. The main new ingredient in their proof is a lemma about the self-correlation of probability distributions in high-dimensional Euclidean space; this lemma can be used to introduce a new kind of density-boost step, which is much simpler to analyze. As such, their proof is also more conceptual, and does not rely on the verification of complicated numerical inequalities. In this section, we discuss their strategy and prove their bound on $r(k)$, while deferring the geometry, and hence the proof of the key lemma, to \cref{sec:geometry}.

\textcite{2410.17197} introduce a modified book algorithm, which removes the asymmetry that was inherent to \cref{alg:book v2}. Roughly speaking, the fact that the colors play symmetric roles is their new approach is the key reason why this argument generalizes to more than two colors, whereas the earlier one of \cite{2303.09521} did not. However, for simplicity, we will continue to only work with two colors, as before. 

\subsection{The steps of the symmetric book algorithm}
In the symmetric book algorithm of \textcite{2410.17197}, we maintain five disjoint sets $A,B,X,Y,Z$, such that $(A,X\cup Y)$ is a red book and $(B,X \cup Z)$ is a blue book. Thus, the only difference from the earlier setup is the new set $Z$, which is a blue analogue of $Y$.
\begin{center}
	\begin{tikzpicture}
		\draw[fill=red, very thick] (0,0) circle[radius=.5] node[] {$A$}; 
		\draw[fill=blue!60!ProcessBlue, very thick] (0,-6) circle[radius=.5] node[] {$B$}; 
		\draw[very thick, fill=white] (-2,-3) ellipse[x radius=1, y radius=.5] node {$X$}; 
		\scoped[on background layer] \fill[red] (-3,-3) -- (0,0) -- (-1.2,-3) -- cycle;
		\scoped[on background layer] \fill[blue!60!ProcessBlue] (-3,-3) -- (0,-6) -- (-1.2,-3) -- cycle;
		\draw[very thick, fill=white] (2,-2) ellipse[x radius=1, y radius=.5] node {$Y$};
		\draw[very thick, fill=white] (2,-4) ellipse[x radius=1, y radius=.5] node {$Z$};
		\scoped[on background layer] \fill[red] (3,-2) -- (0,0) -- (1.2,-2) -- cycle;
		\scoped[on background layer] \fill[blue!60!ProcessBlue] (3,-4) -- (0,-6) -- (1.2,-4) -- cycle;
	\end{tikzpicture}
\end{center}
We initialize the process with $A = B = \varnothing$, and $X \sqcup Y \sqcup Z$ an arbitrary equitable partition of $V(K_N)$. As before, the two main types of steps will be red and blue steps, wherein we move a vertex $v$ from $X$ to $A$ (resp.\ $B$), and shrink $X$ and $Y$ (resp.\ $X$ and $Z$) to their intersection with the red (resp.\ blue) neighborhood of $v$. 

In our prior analysis of the book algorithm, we had to track four parameters: the sizes of $A,B$, and $X$, and the red density between $X$ and $Y$. In this new setup, we will naturally also track $\ab Z$ and the blue density between $X$ and $Z$, which controls how much $Z$ shrinks every time we do a blue step. Let us denote the red and blue densities, respectively, as
\[
	p_R \coloneqq d_R(X,Y) = \frac{e_R(X,Y)}{\ab X \ab Y} \qquad \text{ and } \qquad p_B \coloneqq d_B(X,Z) = \frac{e_B(X,Z)}{\ab X \ab Z}.
\]
In analogy with \cref{ass:p_0,ass:degree regular}, we make the following assumptions about the red and blue densities. As before, one can remove these assumptions by working harder, but by imposing them we can simplify the exposition and focus on the key new ideas.
\begin{assumption}\label{ass:left regular}
	At every step of the process, every vertex in $X$ has exactly $p_R \ab Y$ red neighbors in $Y$ and exactly $p_B \ab Z$ blue neighbors in $Z$. 
\end{assumption}
\begin{assumption}\label{ass:p_0 both colors}
	At the beginning of the process, we have $p_R \geq \frac 12$ and $p_B \geq \frac 12$. 
\end{assumption}
Note that \cref{ass:left regular} is a weaker assumption than \cref{ass:degree regular}, since we do not assume anything about the degrees of vertices in $Y$ or $Z$. In our earlier analysis of the book algorithm, such an assumption was necessary in order to prove \eqref{eq:density-boost density}, the basic inequality governing the contribution of density-boost steps. In what follows, we will not need this inequality, or indeed any of the earlier analysis of density-boost steps. In its place, we will use the following key lemma, whose proof we defer to \cref{sec:geometry}.
\begin{lemma}[Refinement lemma]\label{lem:refinement}
	Let $X,Y,Z$ be disjoint sets of vertices in a coloring of $E(K_N)$, and let $\alpha_R,\alpha_B \in (0,1)$ be parameters. There exists a vertex $v \in X$, a subset $X' \subseteq X$, and a real number $\kappa\geq 0$ with the following properties.
	\begin{lemenum}
		\item \textbf{$X'$ is not much smaller than $X$:} We have $\ab{X'} \geq c 2^{-C\kappa} \ab X$, where $C = 6$ and $c = 1/8$ are constants.\label{it:X' size}
		\item \textbf{No density drop in either color:} Letting $Y' = Y \cap N_R(v)$ and $Z' = Z \cap N_B(v)$, we have that
		\[
			d_R(X',Y') \geq d_R(X,Y) - \alpha_R \qquad \text{ and } \qquad d_B(X',Z') \geq d_B(X,Z) - \alpha_B.
		\]
		\label{it:no density drop}
		
		\item \textbf{Density boost in one color:} We have that
		\[
			d_R(X',Y') \geq d_R(X,Y) + (\kappa^2-1) \alpha_R \qquad\text{ or } \qquad d_B(X',Z') \geq d_B(X,Z) + (\kappa^2-1)\alpha_B.
		\]
		\label{it:density boost}
	\end{lemenum}
\end{lemma}

In the symmetric book algorithm, we apply \cref{lem:refinement} at every iteration, before deciding whether to take a red, blue, or density-boost step. 
As before, the quantities $\alpha_R,\alpha_B$ determine how much we are willing to lose on the red and blue density, respectively. As such, we see that \cref{it:no density drop} is extremely helpful, and appears to completely eliminate the need for density-boost steps: every time we take a red or a blue step, \cref{it:no density drop} guarantees that the relevant density does not drop by more than our specified amount of $\alpha_R$ or $\alpha_B$. 

However, we cannot entirely dispense with density-boost steps. Indeed, every time we apply \cref{lem:refinement}, we shrink $X$ to $X'$, and we have little control over how much smaller $X'$ is than $X$: note that, crucially, $\kappa$ is an \emph{output} of \cref{lem:refinement}, and not an input, so it may be that $X'$ is much smaller than $X$. This is where \cref{it:density boost} comes in: in case $\kappa$ is very large, so that $X'$ is much smaller than $X$, we gain a great deal on the density in one of the two colors. We can thus perform a new kind of density-boost step, replacing $X$ by $X'$ and one of $Y,Z$ by $Y',Z'$. As we will see, the trade-offs in this density-boost step work out in our favor; it is crucial here that $X$ shrinks by a factor of $\exp(-O(\kappa))$, but that the density increases by an amount proportional to $\kappa^2$. Because we only do such density-boost steps when $\kappa$ is large, the amount we gain on the density is substantially larger than the amount we lose on $\ab X$.

One final difference between the symmetric book algorithm and the book algorithm studied in \cref{sec:book alg,sec:rescue} is the stopping condition. We will fix some integers $t,m$ satisfying the assumptions of \cref{lem:book to clique}, and our goal is to produce a monochromatic $B_{t,m}$. As such, we will stop the process if $\ab A \geq t$ or $\ab B \geq t$, and we aim to prove that at the end of the process, we still have $\ab Y \geq m$ and $\ab Z \geq m$. When the process terminates, we have thus found our monochromatic $B_{t,m}$. Note that this stopping condition guarantees that we never do more than $t$ red steps or more than $t$ blue steps.

The two things that remain are to define $\alpha_R, \alpha_B$, and to specify the cutoff for $\kappa$ that determines whether we do a density-boost step or a red or blue step. For the former, we use the same definition as before, only adjusted to the fact that we now do at most $t$ red or blue steps, namely we set $\varepsilon = t^{-1/4}$ and define
\begin{equation}\label{eq:symmetric alpha def}
	\alpha(p) \coloneqq
	\begin{cases}
		\frac{\varepsilon}{t} & \text{if } p \leq \frac 12+\frac 1 t,\\
		\varepsilon(p-\frac 12)&\text{otherwise,}
	\end{cases}
\end{equation}
and then set $\alpha_R \coloneqq \alpha(p_R), \alpha_B \coloneqq \alpha(p_B)$. Finally, the cutoff $\kappa_{\mathrm{cutoff}}$ for $\kappa$ will be a fixed constant.

\subsection{Formal definition of the symmetric book algorithm}
We are now ready to define the symmetric book algorithm of \textcite{2410.17197}. 
We begin by fixing various parameters. First, as in \cref{lem:refinement}, we set $C=6$ and $c=1/8$. We then define
\begin{equation}\label{eq:eta and t def}
	\eta \coloneqq \frac{1}{8000} \qquad \text{ and } \qquad t \coloneqq \eta k,
\end{equation}
and set $\varepsilon = t^{-1/4}$, analogously to our choice of $\varepsilon$ in \cref{sec:book alg}.
Finally, we set
\begin{equation}\label{eq:kappa cutoff def}
	\kappa_{\mathrm{cutoff}} \coloneqq 400=\sqrt{\frac{20}\eta}.
\end{equation}
We remark that in the proof, the only properties we will use of these constants are that
\[
	\kappa_{\mathrm{cutoff}} \geq 8C, \quad \kappa_{\mathrm{cutoff}}^2 \geq 8 \log_2 \frac 1c, \quad 2\left( \log_2 \frac 1c + C\kappa_{\mathrm{cutoff}} \right) +5 < \frac 1 \eta, \quad \text{and}\quad\kappa_{\mathrm{cutoff}}^2 \geq \frac{20}{\eta}.
\]
Clearly any sufficiently large $\kappa_{\mathrm{cutoff}}$ will satisfy the first two constraints, and additionally for any sufficiently large $\kappa_{\mathrm{cutoff}}$ we can pick an appropriate $\eta$ to satisfy the latter two. The final upper bound on $r(k)$ that we will prove (see \cref{thm:balister}) is
$r(k) \leq 2^{(2-\eta^2/10)k}$,
hence we are interested in choosing $\eta$ as large as possible while satisfying the constraints above. However, we remark that no real effort was made to optimize any of the constants.

The symmetric book algorithm is then defined as follows, initialized with $A = B = \varnothing$ and $X \sqcup Y \sqcup Z$ an arbitrary partition of $V(K_N)$ into three sets of size\dangerfoot{} $N/3$.
\begin{algorithm}{Symmetric book algorithm}{symmetric}
	\fixind
	\begin{enumerate}
		\item If $X \leq 1, \ab A \geq t$, or $\ab B \geq t$, stop the process.\label{it:check sizes symm}
		\item Let $p_R = d_R(X,Y)$ and $p_B = d_B(X,Z)$ be the current red and blue densities. Define $\alpha_R = \alpha(p_R)$ and $\alpha_B = \alpha(p_B)$ according to \eqref{eq:symmetric alpha def}.
		\item Perform a \emph{refinement step}: apply \cref{lem:refinement} to obtain a vertex $v \in X$, a set $X' \subseteq X$, and a number $\kappa \geq 0$. Let $Y' = Y \cap N_R(v)$ and $Z' = Z \cap N_B(v)$.
		\item If $\kappa \geq \kappa_{\mathrm{cutoff}}$, perform a \emph{density-boost step} in one of the two colors.\label{it:large kappa}
		\begin{enumerate}
			\item If $d_R(X',Y') \geq (\kappa^2-1) p_R$, perform a \emph{red density-boost step}, by updating
			\[
				A \to A, \qquad B \to B, \qquad X \to X', \qquad Y \to Y', \qquad Z \to Z,
			\]
			and return to step \ref{it:check sizes symm}.
			\item If $d_B(X',Z') \geq (\kappa^2-1) p_B$, perform a \emph{blue density-boost step}, by updating
			\[
				A \to A, \qquad B \to B, \qquad X \to X', \qquad Y \to Y, \qquad Z \to Z',
			\]
			and return to step \ref{it:check sizes symm}.
		\end{enumerate}
		
		\item If $\kappa < \kappa_{\mathrm{cutoff}}$, check whether $v$ has at least $\frac 12 \ab{X'}$ red neighbors in $X'$.
		\begin{enumerate}
			\item If yes, perform a \emph{red step}, by updating
			\[
				A \to A \cup \{v\}, \qquad B \to B, \qquad X \to X' \cap N_R(v), \qquad Y \to Y', \qquad Z \to Z,
			\]
			and return to step \ref{it:check sizes symm}.
			
			\item If not, then $v$ has at least\dangerfoot{} $\frac 12 \ab{X'}$ blue neighbors in $X'$. In this case, perform a \emph{blue step}, by updating
			\[
				A \to A, \qquad B \to B \cup \{v\}, \qquad X \to X' \cap N_B(v), \qquad Y \to Y, \qquad Z \to Z',
			\]
			and return to step \ref{it:check sizes symm}.
		\end{enumerate}
	\end{enumerate}
\end{algorithm}
Note that by \cref{ass:left regular}, we have that $\ab{Y'}=p_R \ab Y$ and $\ab{Z'}=p_B \ab Z$. Additionally, whenever we perform a red step or a red density-boost step, we shrink $X$ to a subset and do not change $Z$, hence we do not affect $p_B$; similarly, we do not affect $p_R$ when doing a red step or red density-boost step. Using these facts, we can fill out \cref{table:symm book}, which records how the key parameters change during the execution of \cref{alg:symmetric}.
\begin{table}[ht!]
	\begin{center}
		\tabulinesep=1.2mm
		\begin{tabu}{|c||c|c|c|c|c|c|c|}
			\hline
			&$\ab A$ & $\ab B$ & $\ab X$ & $\ab Y$ & $\ab Z$ & $p_R$ & $p_B$ \\ \hhline{|=#*{7}{=|}}
			red density-boost & -- & -- & $\times c2^{-C\kappa}$ & $\times p_R$ & -- & $+ (\kappa^2-1)\alpha_R$ & --\\ \hline
			blue density-boost & -- & -- & $\times c2^{-C\kappa}$ & -- & $\times p_B$ & -- & $+ (\kappa^2-1)\alpha_B$\\ \hline
			red step & $+1$ & -- & $\times \frac 12c2^{-C\kappa}$ & $\times p_R$ & -- & $-\alpha_R$ & -- \\ \hline
			blue step & -- & $+1$ & $\times \frac 12c2^{-C\kappa}$ & -- & $\times p_B$ & -- & $-\alpha_B$ \\ \hline
		\end{tabu}
	\end{center}
	\caption{How the various parameters evolve during \cref{alg:symmetric}.
	}
	\label{table:symm book}
\end{table}

\subsection{Analysis of the symmetric book algorithm}
We now collect various lemmas about how the development of the parameters throughout the execution of \cref{alg:symmetric}. Let $s_R$ and $s_B$ denote the total number of red and blue density-boost steps done during the process. Recall that by the stopping condition in step \ref{it:check sizes symm}, we do at most $t$ red and at most $t$ blue steps. 

The first two lemmas are proved identically to \cref{lem:p never drops,lem:Y size}, respectively.
\begin{lemma}
	We have $p_R \geq \frac 12 - \varepsilon$ and $p_B \geq \frac 12- \varepsilon$ throughout the entire process.
\end{lemma}
\begin{lemma}\label{lem:YZ sizes}
	At the end of the process, we have $\ab Y \geq 2^{-t-s_R-o(k)} N$ and $\ab Z \geq 2^{-t-s_B-o(k)} N$.
\end{lemma}
Now, let $\kappa_{R,i}, \kappa_{B,i}$ denote the value of $\kappa$ used during the $i$th red (resp.\ blue) density-boost step. Additionally, let $\kappa_R, \kappa_B$ denote the average value of these numbers, i.e.\
\[
	\kappa_R \coloneqq \frac{1}{s_R} \sum_{i=1}^{s_R} \kappa_{R,i} \qquad\text{ and } \qquad \kappa_B \coloneqq \frac{1}{s_B} \sum_{i=1}^{s_B} \kappa_{B,i}.
\]
The next lemma is analogous to \cref{lem:X size}.
\begin{lemma}\label{lem:X size symmetric}
	At the end of the process, we have
	\[
		\ab X \geq \left( \frac 12 c2^{-C \kappa_{\mathrm{cutoff}}} \right)^{2t+o(k)} c^{s_R+s_B} 2^{-C(s_R \kappa_R + s_B \kappa_B)}N.
	\]
\end{lemma}
\begin{proof}
	Every red or blue step shrinks $X$ by at most a factor\dangerfoot{} of $\frac 12 c2^{-C \kappa_{\mathrm{cutoff}}}$, since we only do a red or blue step if the current value of $\kappa$ is at most $\kappa_{\mathrm{cutoff}}$. As we do at most $t$ red and at most $t$ blue steps, we obtain the factor of $(\frac 12 c2^{-C \kappa_{\mathrm{cutoff}}})^{2t}$. Similarly, the $i$th red density-boost step shrinks $X$ by at most a factor of $c2^{-C\kappa_{R,i}}$, hence the total contribution of the red density-boost steps is
	\[
		\prod_{i=1}^{s_R} c2^{-C \kappa_{R,i}} = c^{s_R} 2^{-C \sum_{i=1}^{s_R} \kappa_{R,i}} =c^{s_R} 2^{-C s_R \kappa_R},
	\]
	by the definition of $\kappa_R$. Adding in the analogous contribution of the blue density-boost steps yields the claimed result.
\end{proof}

Finally, the following lemma is analogous to \cref{lem:zigzag}, and is proved identically. Indeed, the only difference from \cref{lem:zigzag} is that, rather than saying that the $i$th density-boost step boosts the red density by $\frac{1-\beta_i}{\beta_i}\alpha$, we say that the $i$th red density-boost step increases $p_R$ by at least $(\kappa_{R,i}^2-1)\alpha_R$, and similarly for the blue density-boost steps.
\begin{lemma}\label{lem:symmetric zigzag}
	We have
	\[
		\sum_{i=1}^{s_R} (\kappa_{R,i}^2-1) \leq t+o(k) \qquad \text{ and } \qquad \sum_{i=1}^{s_B} (\kappa_{B,i}^2-1) \leq t+o(k).
	\]
\end{lemma}
We remark that in the analysis of the original book algorithm, it was critical to prove \cref{lem:zigzag} with a main term of $t$, rather than a weaker bound such as $2t$. There is enough slack in the present argument so that a much weaker bound than \cref{lem:symmetric zigzag} would suffice (and indeed \textcite[Lemma 4.3]{2410.17197} prove such a weaker bound, with a constant-factor loss).

Using \cref{lem:symmetric zigzag}, we can prove the following bound on the number of density-boost steps and on their total contribution to \cref{lem:X size symmetric}, which will play the same role in the argument as \cref{lem:beta LB}.
\begin{lemma}\label{lem:kappa contribution}
	For sufficiently large $k$, we have
	\[
		s_R \kappa_R \leq \frac{4t}{\kappa_{\mathrm{cutoff}}} \qquad \text { and } \qquad s_B \kappa_B \leq \frac{4t}{\kappa_{\mathrm{cutoff}}}.
	\]
	Additionally,
	\[
		s_R \leq \frac{4t}{\kappa_{\mathrm{cutoff}}^2} \qquad \text{ and } \qquad s_B \leq \frac{4t}{\kappa_{\mathrm{cutoff}}^2}.
	\]
\end{lemma}
\begin{proof}
	Recall from \eqref{eq:eta and t def} that $t = \eta k$ for a fixed constant $\eta$. Hence, for sufficiently large $k$, the $t+o(k)$ term in \cref{lem:symmetric zigzag} is at most $2t$. We henceforth assume that $k$ is sufficiently large for this to hold. Next, we recall by step \ref{it:large kappa} of \cref{alg:symmetric} that $\kappa_{R,i} \geq \kappa_{\mathrm{cutoff}}$ for all $i$, and that $\kappa_{\mathrm{cutoff}} \geq \sqrt 2$ by our choice in \eqref{eq:kappa cutoff def}. Therefore,
	\[
		2t \geq \sum_{i=1}^{s_R}(\kappa_{R,i}^2-1) \geq \frac 12 \sum_{i=1}^{s_R} \kappa_{R,i}^2 \geq \frac{\kappa_{\mathrm{cutoff}}}{2} \sum_{i=1}^{s_R} \kappa_{R,i} = \frac{\kappa_{\mathrm{cutoff}}}{2} \cdot (s_R \kappa_R),
	\]
	where we use \cref{lem:symmetric zigzag} in the first inequality, that $x^2-1 \geq \frac 12 x^2$ for $x \geq \sqrt 2$ in the second, that $\kappa_{R,i} \geq \kappa_{\mathrm{cutoff}}$ in the third, and the definition of $\kappa_R$ in the fourth. This, together with the analogous result for blue, yields the first claimed result.

	The second claimed result then follows immediately if we use once more that $\kappa_R \geq \kappa_{\mathrm{cutoff}}$, and similarly for blue.
\end{proof}
Recall from \cref{lem:X size symmetric} that the total contribution of the density-boost steps to $\ab X$ is $c^{s_R + s_B}e^{-C(s_R \kappa_R + s_B \kappa_B)}$. By \cref{lem:kappa contribution}, both of these exponential terms can be bounded as $\exp(O(t/\kappa_{\mathrm{cutoff}}))$. This means that if $\kappa_{\mathrm{cutoff}}$ is sufficiently large in terms of $c$ and $C$---and our choice in \eqref{eq:kappa cutoff def} is indeed sufficiently large---then the contribution of the density-boost steps is negligible compared to the other exponential terms. 

It is here that we see why it is crucial to obtain in \cref{lem:refinement} a density boost which asymptotically dominates $\kappa$. It is thanks to this that the proof of \cref{lem:kappa contribution} gives us a $\kappa_{\mathrm{cutoff}}$ term in the denominator; thanks to this we can make the density-boost steps negligible by a sufficiently large choice of $\kappa_{\mathrm{cutoff}}$. That is, if \cref{lem:refinement} only gave a density-boost of order $\kappa$, this argument would not work. Luckily \cref{lem:refinement} boosts the density by a quadratic quantity in $\kappa$, which makes the whole argument go through.

\subsection{Proof of the upper bound on $r(k)$}
Given our lemmas above on the evolution of the symmetric book algorithm, it is fairly straightforward to prove the following result, which implies that \cref{alg:symmetric} always outputs a sufficiently large monochromatic book.
\begin{lemma}\label{lem:symm alg finds book}
	Let $k$ be sufficiently large, and let
	\[
		N \geq \max \{2^{t+\eta t/4} m, 2^{k}\}
	\]
	for some integer $m$. Then \cref{alg:symmetric} terminates by finding a monochromatic copy of $B_{t,m}$.
\end{lemma}
\begin{proof}
	It suffices to prove that at the end of the process, we have $\ab X>1$ and $\ab Y, \ab Z \geq m$. Indeed, if this is the case, then \cref{alg:symmetric} can only terminate when one of $A$ or $B$ has size $t$, at which point it forms a monochromatic copy of $B_{t,m}$ with $Y$ or $Z$, respectively.

	For the second claim, we have by \cref{lem:YZ sizes} that $\ab Y\geq 2^{-t-s_R-o(k)}N$. By \cref{lem:kappa contribution}, we have that
	\[
		s_R \leq \frac{4t}{\kappa_{\mathrm{cutoff}^2}} = \frac{\eta t}{5},
	\]
	since $\kappa_{\mathrm{cutoff}}^2= 20/\eta$ by \eqref{eq:kappa cutoff def}. Therefore, for $k$ sufficiently large we have that $s_R+o(k) \leq \eta t/4$, hence $\ab Y \geq 2^{-t-\eta t/4}N \geq m$. The proof that $\ab Z \geq m$ at the end of the process proceeds identically.

	It remains to prove the lower bound on $\ab X$, for which we use \cref{lem:X size symmetric}. We estimate each of the quantities appearing in it in turn. First, recalling the definitions of $c$, $C$, and $\kappa_{\mathrm{cutoff}}$, we see that
	\[
		\frac 12 c 2^{-C \kappa_{\mathrm{cutoff}}} = 2^{-4} \cdot 2^{-6 \cdot 400} = 2^{-2404}.
	\]
	Therefore, for $k$ sufficiently large, we have that 
	\[
		\left(\frac 12 c 2^{-C \kappa_{\mathrm{cutoff}}}\right)^{2t+o(k)} \geq 2^{-6000t}.
	\]
	Next, by \cref{lem:kappa contribution}, we have that
	\[
		c^{s_R+s_B} = 2^{-3(s_R+s_B)} \geq 2^{-24t/\kappa_{\mathrm{cutoff}}^2}\geq 2^{-t},
	\]
	since $24/\kappa_{\mathrm{cutoff}}^2 = 24/400^2 \leq 1$. Finally, again by \cref{lem:kappa contribution}, we have that
	\[
		C(s_R \kappa_R+s_B \kappa_B) \leq \frac{8Ct}{\kappa_{\mathrm{cutoff}}} = \frac{48}{400}t \leq t.
	\]
	Plugging this all into \cref{lem:X size symmetric}, we conclude that at the end of the process,
	\[
		\ab X \geq \left( \frac 12 c2^{-C \kappa_{\mathrm{cutoff}}} \right)^{2t+o(k)} c^{s_R+s_B} 2^{-C(s_R \kappa_R + s_B \kappa_B)}N \geq 2^{-6000t-t-t} N > 2^{-t/\eta}N = 2^{-k}N,
	\]
	since $6002 < 1/\eta$ and $k = \eta t$. Finally, by our assumption that $N \geq 2^k$, we conclude that $\ab X >1$ at the end of the process, as claimed.
\end{proof}
Finally, the upper bound on diagonal Ramsey numbers is an immediate consequence of \cref{lem:symm alg finds book,lem:book to clique}. 
\begin{theorem}[\cite{2410.17197}]\label{thm:balister}
	If $k$ is sufficiently large, then 
	\[r(k) \leq 2^{(2-\eta^2/10)k}.\]
\end{theorem}
We stress that, as always, the proof we give is not complete, as it relies on \cref{ass:left regular,ass:p_0 both colors}, both of which we have not justified. Additionally, we remark that this is not the strongest bound that can be proved by this technique; in particular, the choices of $\eta$ and $\kappa_{\mathrm{cutoff}}$ are not fully optimized.
\begin{proof}[Proof of \cref{thm:balister}]
	Let $k$ be sufficiently large, and let $N = 2^{(2-\eta^2/10)k}$. Fix a two-coloring of $E(K_N)$, and run \cref{alg:symmetric}. Recalling that $t=\eta k$, let $m = r(k-t,k)$. By \cref{thm:ES binomial}, we have that
	\[
		m \leq \binom{2k-t}k = \binom{(2-\eta)k}k \leq 2^{(2-\eta) H(\frac{1}{2-\eta})\cdot k}.
	\]
	Now, we have that
	\[
		\left(x + \frac{x^2}{4}\right) + (2-x) H\left(\frac{1}{2-x}\right)  \leq 2-\frac{x^2}{10}
	\]
	for all $x \in [0,1]$; this can be verified by computing the Taylor series of the difference, and noting that all coefficients are non-positive. Applying this fact with $x=\eta$, we conclude that
	\[
		2^{t+\eta t/4} m \leq 2^{(\eta + \eta^2/4)\cdot k} \cdot 2^{(2-\eta) H(\frac{1}{2-\eta})\cdot k} \leq 2^{(2-\eta^2/10)k} = N.
	\]
	Since we also have $N \geq 2^k$, \cref{lem:symm alg finds book} implies that the coloring contains a monochromatic $B_{t,m}$. By \cref{lem:book to clique}, we conclude that the coloring also contains a monochromatic $K_t$, as claimed.
\end{proof}

\section{High-dimensional geometry and the proof of Lemma \ref{lem:refinement}}\label{sec:geometry}
\subsection{Reduction to a geometric statement}
All that remains now is to prove \cref{lem:refinement}, which is really the heart of the proof presented above. As previously mentioned, \cref{lem:refinement} is a consequence of a purely geometric lemma, about probability distributions in high-dimensional Euclidean space. Before stating this geometric lemma, however, let us build up to it by stating a probabilistic strengthening of \cref{lem:refinement}.

For a vertex $v \in X$, let us denote by $N_Y(v)$ the red neighborhood of $v$ in $Y$, that is, $N_Y(v)\coloneqq Y \cap N_R(v)$. Similarly, $N_Z(v) \coloneqq Z \cap N_B(v)$ will denote the blue neighborhood of $v$ in $Z$. With this notation, we now state the probabilistic strengthening of \cref{lem:refinement}.
\begin{lemma}\label{lem:probabilistic}
	Let $X,Y,Z$ be disjoint sets of vertices in a coloring of $E(K_N)$, and let $\alpha_R, \alpha_B \in (0,1)$. Let $p_R = d_R(X,Y)$ and $p_B = d_B(X,Z)$. Let $\bm v,\bm w$ be independent, uniformly random vertices of $X$. There exists a real number $\kappa \geq 0$ such that
	\begin{multline}\label{eq:prob outcome}
		\pr\big(\ab{N_Y(\bm v) \cap N_Y(\bm w)} \geq (p_R+(\kappa^2-1)\alpha_R)p_R\ab Y \\
		\text{and}\qquad \ab{N_Z(\bm v) \cap N_Z(\bm w)} \geq (p_B-\alpha_B)p_B \ab Z\big) \geq c 2^{-C\kappa},
	\end{multline}
	or else the same holds upon swapping the roles of red and blue and of $Y$ and $Z$. Here, $c=1/8$ and $C=6$ are the same constants as in \cref{lem:refinement}.
\end{lemma}
Note that \cref{lem:probabilistic} is a more symmetric statement than \cref{lem:refinement}: it discusses the common red and blue neighborhoods of a random pair of vertices in $X$, as opposed to isolating a special vertex $v \in X$ as well as a special subset $X' \subseteq X$. Nonetheless, \cref{lem:refinement} follows from \cref{lem:probabilistic} by a simple averaging argument.
\begin{proof}[Proof of \cref{lem:refinement}, assuming \cref{lem:probabilistic}]
	Let us suppose that \eqref{eq:prob outcome} holds;
	the other case, where the roles of the colors are swapped, follows by a symmetric argument. First, by the law of total probability, there exists some fixed $v \in X$ such that
	\begin{multline}\label{eq:fixed v}
		\pr\big(\ab{N_Y(v) \cap N_Y(\bm w)} \geq (p_R+(\kappa^2-1)\alpha_R)p_R\ab Y \\
		\text{and}\qquad \ab{N_Z(v) \cap N_Z(\bm w)} \geq (p_B-\alpha_B)p_B \ab Z\big) \geq c 2^{-C\kappa},
	\end{multline}
	This is the same inequality as \eqref{eq:prob outcome}, except that now only $\bm w$ is random. As in the statement of \cref{lem:refinement}, let $Y' = N_Y(v)$ and $Z' = N_Z(v)$. Let $X' \subseteq X$ denote the set of all $w \in X$ such that
	\[
		\ab{N_Y(v) \cap N_Y(w)} \geq (p_R+(\kappa^2-1)\alpha_R)p_R\ab Y \;\; \text{and}\;\; \ab{N_Z(v) \cap N_Z(w)} \geq (p_B-\alpha_B)p_B \ab Z.
	\]
	Then \eqref{eq:fixed v} is equivalent to the statement that $\ab{X'} \geq c2^{-C \kappa} \ab X$, proving \cref{it:X' size}. Next, we have that
	\[
		e_R(X',Y') = \sum_{w \in X'} \ab{N_R(w) \cap Y'} = \sum_{w \in X'} \ab{N_Y(v) \cap N_Y(w)} \geq \ab{X'}\cdot(p_R+(\kappa^2-1)\alpha_R)p_R\ab Y ,
	\]
	by the definition of $X'$. Recalling that $\ab{Y'}=p_R \ab Y$ by \cref{ass:left regular}, we find that
	\[
		d_R(X',Y') = \frac{e_R(X',Y')}{\ab{X'}\ab{Y'}} \geq p_R + (\kappa^2-1)\alpha_R,
	\]
	which proves \cref{it:density boost} and also implies the first bound in \cref{it:no density drop}. Similarly, we can compute that
	\[
		d_B(X',Z') \geq p_B-\alpha_B,
	\]
	which completes the proof of \cref{it:no density drop}.
\end{proof}

It thus only remains to prove \cref{lem:probabilistic}, which follows from a geometric argument. In order to convert the statement of \cref{lem:probabilistic} to a geometric one, we simply observe that the size of the intersection of two sets can be encoded in a linear-algebraic way, as the inner product of their indicator vectors. Namely, let us denote by $\bm 1_Y(v) \in \R^Y$ the indicator vector of the set $N_Y(v)$; this is the vector whose $y$th coordinate equals $1$ if $y \in N_Y(v)$, and $0$ otherwise. Then we have the identity
\[
	\inn{\bm 1_Y(v), \bm 1_Y(w)} = \ab{N_Y(v) \cap N_Y(w)},
\]
since the inner product of two binary vectors is precisely the number of coordinates in which they both have a $1$. 

In order to precisely arrive at the setup of \cref{lem:probabilistic}, we center and rescale the vectors $\bm 1_Y(v)$. Namely, letting $\bm 1$ denote the all-ones vector, we define
\begin{equation}\label{eq:sigma def}
	\sigma_Y(v) \coloneqq \frac{\bm 1_Y(v)-p_R \bm 1}{\sqrt{\alpha_R p_R \ab Y}}.
\end{equation}
Note that by \cref{ass:left regular}, $\bm 1_Y(v)$ has exactly $p_R \ab Y$ entries equal to $1$, hence by subtracting $p_R\bm 1$ we ensure that the sum of the entries in $\sigma_Y(v)$ is zero. Thus, all the vectors $\{\sigma_Y(v) : v\in X\}$ lie on a sphere of radius $\sqrt{\frac{1-p_R}{\alpha_R}}$ centered at the origin in $\R^Y$. Additionally, we can compute that for all $v,w \in X$, we have
\begin{align*}
	\inn{\sigma_Y(v),\sigma_Y(w)} &= \frac{1}{\alpha_R p_R \ab Y} \inn{\bm 1_Y(v)-p_R \bm 1, \bm 1_Y(w) - p_R \bm 1}\\
	&=\frac{1}{\alpha_R p_R \ab Y} \left( \inn{\bm 1_Y(v),\bm 1_Y(w)} - p_R \inn{\bm 1_Y(v), \bm 1} - p_R \inn{\bm 1,\bm 1_Y(w)} + p_R^2 \inn{\bm 1,\bm 1} \right)\\
	&=\frac{1}{\alpha_R p_R \ab Y} \left( \ab{N_Y(v) \cap N_Y(w)}  - p_R^2 \ab Y\right),
\end{align*}
where we used \cref{ass:left regular} in the final step, to conclude that $\inn{\bm 1_Y(v),\bm 1}=\inn{\bm 1,\bm 1_Y(w)}=p_R\ab Y$. 
Note that if $N_Y(v), N_Y(w)$ were random subsets of $Y$, each of size $p_R \ab Y$, then their expected intersection size would be exactly $p_R^2 \ab Y$, hence the inner product $\inn{\sigma_Y(v),\sigma_Y(w)}$ measures (up to scaling) how much their true intersection size deviates from this average value.
In particular, for any $\lambda \in \R$, we have that $\inn{\sigma_Y(v),\sigma_Y(w)} \geq \lambda$ if and only if
\[
	\frac{1}{\alpha_R p_R \ab Y} \left( \ab{N_Y(v) \cap N_Y(w)}  - p_R^2 \ab Y\right) \geq \lambda,
\]
which in turn happens if and only if 
\[
	\ab{N_Y(v) \cap N_Y(w)} \geq \left( p_R + \lambda \alpha_R \right) p_R \ab Y.
\]
This shows that \cref{lem:probabilistic} is a special case of the following geometric result of \textcite{2410.17197}.
\begin{theorem}[Geometric lemma]\label{lem:geometric}
	Let $X,Y,Z$ be finite sets, and let $\sigma_Y:X \to \R^Y$ and $\sigma_Z:X \to \R^Z$ be arbitrary functions. Let $\bm v, \bm w$ be independent, uniformly random elements of $X$. There exists a real number $\kappa\geq 0$ such that
	\begin{equation}\label{eq:prob outcome geom}
		\pr\big(\inn{\sigma_Y(\bm v), \sigma_Y(\bm w)} \geq \kappa^2-1 \qquad\text{and}\qquad \inn{\sigma_Z(\bm v), \sigma_Z(\bm w)} \geq -1\big) \geq c 2^{-C\kappa},
	\end{equation}
	or else the same holds upon swapping the roles of $Y$ and $Z$. 
\end{theorem}
Indeed, the statement of \cref{lem:probabilistic} is precisely that of \cref{lem:geometric} when one defines $\sigma_Y$ as in \eqref{eq:sigma def} (and analogously for $\sigma_Z$). Hence it only remains to prove \cref{lem:geometric}.

Before turning to the proof, a few remarks about \cref{lem:geometric} are in order. First, note that the sets $Y$ and $Z$ play no role in the statement: we could equally well take $\sigma_Y$ and $\sigma_Z$ to be valued in the same high-dimensional Euclidean space $\R^n$, or alternately in the infinite-dimensional Hilbert space $\ell^2$. Second, as we will see, the set $X$ also does not really matter: the same statement holds for arbitrary coupled probability distributions on $\R^Y \times \R^Z$, and not only those distributions that are uniform on a finite support. Finally, and perhaps most importantly, we stress that the statement of \cref{lem:geometric} is \emph{not} scale-invariant, unlike many other results about inner products. That is, because the outcome of \cref{lem:geometric} involves constants like $-1$, the statement of the lemma changes if we replace, say, $\sigma_Y$ by $100\sigma_Y$: we cannot simply take the same value of $\kappa$ upon such a rescaling. In fact, this lack of scale-invariance is crucial in the deduction of \cref{lem:probabilistic} and hence \cref{lem:refinement}: recall from \eqref{eq:sigma def} that we chose a careful scaling of $\sigma_Y$, in order to incorporate the parameter $\alpha_R$. 

\subsection{The one-color version of \cref{lem:geometric}}
\cref{lem:geometric} is a ``two-color'' statement, in that it involves two functions $\sigma_Y,\sigma_Z$, as is necessary for the application to two-color Ramsey numbers. \textcite[Lemma 3.1]{2410.17197} prove a more general result, which gives essentially the same outcome but for an arbitrary number of functions $\sigma_1,\dots,\sigma_q$; this is, naturally, necessary for their application to upper bounds on multicolor Ramsey numbers. While we will continue to work in the two-color setting, the proof of \cref{lem:geometric} is easier to understand if one begins by proving a simpler statement, which is the \emph{one-color} version of this result.
\begin{proposition}[One-color geometric lemma]\label{lem:one color}
	Let $X,Y$ be finite sets, and let $\sigma:X \to \R^Y$ be an arbitrary function. Let $\bm v, \bm w$ be independent, uniformly random elements of $X$. There exists a real number $\kappa \geq 0$ such that
	\[
		\pr \left( \inn{\sigma(\bm v), \sigma(\bm w)} \geq \kappa^2-1 \right) \geq \frac{1}{(\kappa^2+2)^2}.
	\]
\end{proposition}
Note that the quantitative dependencies in the one-color statement are much better: we obtain a \emph{polynomial} relationship between the lower bound on the inner products and the success probability, rather than the exponential relationship appearing in \cref{lem:geometric}. It remains open whether there is such a polynomial strengthening of \cref{lem:geometric}; if there is, one would obtain a strengthened version of \cref{lem:refinement}, which would in all likelihood lead to better upper bounds on diagonal Ramsey numbers.

In the proof of \cref{lem:one color}, we will use two simple and well-known observations. The first states that the expected inner product of independent, identically distributed random vectors is non-negative.
\begin{lemma}\label{obs:expected inner prod}
	Let $\bm \sigma, \bm \sigma'$ be independent, identically distributed random vectors in $\R^n$, sampled from an arbitrary probability distribution. Then
	\[
		\E[\inn{\bm \sigma, \bm \sigma'}] \geq 0.
	\]
\end{lemma}
\begin{proof}
	By linearity of expectation and the independence of $\bm \sigma$ and $\bm \sigma'$, we have
	\[
		\E[\inn{\bm \sigma, \bm \sigma'}] = \inn{\E[\bm \sigma], \E[\bm \sigma']}.
	\]
	$\E[\bm \sigma]$ is some fixed vector in $\R^n$, and as $\bm \sigma$ and $\bm \sigma'$ are identically distributed, it equals $\E[\bm \sigma']$. Hence $\E[\inn{\bm \sigma, \bm \sigma'}]$ is simply the squared norm of some vector in $\R^n$, which is non-negative.
\end{proof}
The second well-known observation we need is an immediate consequence of Fubini's theorem.
\begin{lemma}\label{obs:fubini}
	For any random variable $\bm V$, we have
	\[
		\E[\bm V] \leq \int_0^\infty \pr(\bm V \geq \lambda) \dd \lambda.
	\]
\end{lemma}
\begin{proof}
	Let the underlying probability space of $\bm V$ be $(\Omega,\mu)$. Then $\pr(\bm V \geq \lambda)=\int_\Omega \bm 1_{\bm V(\omega)\geq \lambda} \dd \mu(\omega)$. Additionally, for any fixed real number $V$, we have $\int_0^\infty \bm 1_{V \geq \lambda} \dd \lambda$ equals $V$ if $V \geq 0$, and equals $0$ otherwise. In particular, for all $V$, we have $\int_0^\infty \bm 1_{V \geq \lambda} \dd \lambda\geq V$. Therefore, by Fubini's theorem,
	\begin{align*}
		\int_0^\infty \pr(\bm V \geq \lambda) \dd \lambda &= \int_0^\infty \int_\Omega \bm 1_{\bm V(\omega) \geq \lambda} \dd \mu(\omega) \dd \lambda \\
		&= \int_\Omega \int_0^\infty \bm 1_{\bm V(\omega) \geq \lambda} \dd \lambda \dd \mu(\omega)\\
		&\geq\int_\Omega \bm V(\omega) \dd \mu(\omega)\\
		&=\E[\bm V].\qedhere
	\end{align*}

\end{proof}

\cref{lem:one color} is now a simple consequence of these two observations.
\begin{proof}[Proof of \cref{lem:one color}]
	Let $\bm V = \inn{\sigma(\bm v), \sigma(\bm w)} +1$, so that $\E[\bm V] \geq 1$ by \cref{obs:expected inner prod}. By \cref{obs:fubini}, we find that
	\[
		1 \leq \E[\bm V] \leq \int_0^\infty \pr(\bm V \geq \lambda) \dd \lambda = \int_{-1}^\infty \pr(\inn{\sigma(\bm v), \sigma(\bm w)} \geq \lambda) \dd \lambda.
	\]
	If we had that $\pr(\inn{\sigma(\bm v), \sigma(\bm w)} \geq \lambda) \leq 1/(\lambda+3)^2$ for all $\lambda \geq -1$, then we would conclude that
	\[
		1\leq \int_{-1}^\infty \pr(\inn{\sigma(\bm v), \sigma(\bm w)} \geq \lambda) \dd \lambda \leq \int_{-1}^\infty\frac{1}{(\lambda+3)^2} \dd \lambda = \frac 12,
	\]
	a contradiction. Therefore, there exists some $\lambda \geq -1$ for which $\pr(\inn{\sigma(\bm v), \sigma(\bm w)} \geq \lambda) \geq 1/(\lambda+3)^2$. 
	Substituting $\kappa = \sqrt{\lambda+1}$, we conclude that
	\[
		\pr(\inn{\sigma(\bm v), \sigma(\bm w)} \geq \kappa^2-1) \geq \frac{1}{(\kappa^2+2)^2},
	\]
	as claimed.
\end{proof}

\subsection{Proof of \cref{lem:geometric}}
Although the proof of \cref{lem:one color} was very simple, it is helpful to reconsider it from a high-level view, as a very similar strategy will be used to prove \cref{lem:geometric}. In the proof of \cref{lem:one color}, we considered the function $f(x)=x+1$, and then defined a random variable
\[
	\bm V \coloneqq f(\inn{\sigma(\bm v), \sigma(\bm w)}).
\]
The proof of \cref{lem:one color} relies on three important properties of this function $f$, although they are all so simple that they were not explicitly noted in the proof. First, $f$ is non-positive on the region $\{x:x \leq -1\}$; this was implicitly used to restrict our attention to $\lambda \geq -1$, which ends up meaning that we obtain $\kappa \geq 0$. Second, $f$ does not grow too fast, which was used to obtain a good trade-off between $\lambda$ and the probability $\pr(\bm V \geq \lambda)$. Finally, and most importantly, we had the inequality $\E[\bm V] \geq 1$, which followed from \cref{obs:expected inner prod}.

In the proof of \cref{lem:geometric}, we will pick a more complicated function $f:\R^2 \to \R$, and use it to define a random variable $\bm V$ by
\begin{equation}\label{eq:V def}
	\bm V \coloneqq f(\inn{\sigma_Y(\bm v), \sigma_Y(\bm w)}, \inn{\sigma_Z(\bm v), \sigma_Z(\bm w)}).
\end{equation}
In addition to requiring $f$ to satisfy analogues of the three properties discussed above, we also want the choice of $f$ to allow us
to relate $\pr(\bm V \geq \lambda)$ 
to the probability we wish to estimate in \eqref{eq:prob outcome geom}. 
The main difficulty, of course, is that we now want to ensure that two inner products are simultaneously large.

To that end, we define
\begin{equation}\label{eq:M def}
	\bm M \coloneqq \max \{\inn{\sigma_Y(\bm v), \sigma_Y(\bm w)}, \inn{\sigma_Z(\bm v), \sigma_Z(\bm w)}\}
\end{equation}
and let $\cE$ be the event
\begin{equation}\label{eq:E def}
	\cE \coloneqq \left\{ \inn{\sigma_Y(\bm v), \sigma_Y(\bm w)} \geq -1 \qquad \text{ and } \qquad \inn{\sigma_Z(\bm v), \sigma_Z(\bm w)}\geq -1 \right\}.
\end{equation}
Then our first goal is to obtain a lower bound on $\pr(\{\bm M \geq \kappa^2-1\} \cap \cE)$, for some $\kappa$; such a lower bound can be directly converted to a lower bound on the probability in \eqref{eq:prob outcome geom} (or on the analogous quantity upon interchanging the roles of $Y$ and $Z$). As such, it is natural to try defining the random variable $\bm V$ as $\bm M \bm 1_\cE$, since then the probability we are interested in is precisely of the form $\pr(\bm V \geq \lambda)$. The problem with this choice is that it is not clear how to lower-bound $\E[\bm V]$, or indeed why one should expect any non-trivial lower bound to hold. Indeed, while $\E[\bm M] \geq 0$ by \cref{obs:expected inner prod}, it is possible that much of the positive contribution to $\E[\bm M]$ is attained on the complement of $\cE$.

Thus, in order for $\bm V$ to ``know about'' the event $\cE$, we will ensure that $f(y,z) \leq 0$ if $y \leq -1$ or $z \leq -1$. 
This ensures that $\bm V$ is non-positive on the complement of $\cE$, so that bounds on $\pr(\bm V \geq \lambda)$ can be converted to bounds on the quantity we are interested in, $\pr(\{\bm M \geq \kappa^2-1\} \cap \cE)$. Next, as before, we want a global upper bound on $f(y,z)$ in terms of $y$ and $z$, and we would like this upper bound to be ``reasonable'': we will use this bound to convert information of the form $\bm V \geq \lambda$ to information of form $\bm M \geq \kappa^2-1$, and we want good control on the dependencies between $\kappa$ and $\lambda$. Finally, and most importantly, we want that $\E[\bm V] \geq 1$, so that we can apply the same argument as in the proof of \cref{lem:one color}.

This last condition seems like the most difficult one to satisfy, since we have very little information about the random variables $\inn{\sigma_Y(\bm v), \sigma_Y(\bm w)}$ and $\inn{\sigma_Z(\bm v), \sigma_Z(\bm w)}$: if $f$ is some complicated function of these variables, why should we expect $\E[\bm V] \geq 1$? 

The way to ensure this is via a very simple but remarkably powerful trick, as follows. First, we will ensure that $f$ is analytic with $f(0,0)=1$. We then consider the Taylor series of $f$, that is, we write $f(y,z) = \sum_{a,b \geq 0} r_{a,b} y^a z^b$, and apply analyticity and linearity of expectation to conclude that
\[
	\E[\bm V] = \sum_{a,b = 0}^\infty r_{a,b} \E[\inn{\sigma_Y(\bm v), \sigma_Y(\bm w)}^a \inn{\sigma_Z(\bm v), \sigma_Z(\bm w)}^b].
\]
The utility of this expansion comes from the following simple lemma, a 
generalization of \cref{obs:expected inner prod}, which states that all the expectations appearing in the expression above are non-negative. Thus, a lower bound on $\E[\bm V]$ follows immediately, as long as we ensure that $f$ has only non-negative Taylor coefficients.
\begin{lemma}\label{lem:moments positive}
	Fix any probability distribution on $\R^Y \times \R^Z$, and let $(\bm \sigma_Y, \bm \sigma_Z)$ and $(\bm \sigma_Y', \bm \sigma_Z')$ be independent samples from this distribution. We have
	\[
		\E[\inn{\bm \sigma_Y, \bm \sigma_Y'}^a \inn{\bm \sigma_Z, \bm \sigma_Z'}^b] \geq 0
	\]
	for all non-negative integers $a,b$.
\end{lemma}
\begin{proof}
	Let $\bm \sigma$ be a random vector in $(\R^Y)^{\otimes a} \otimes (\R^Z)^{\otimes b}$ defined by
	\[
		\bm \sigma = (\underbrace{\bm \sigma_Y \otimes \dots \otimes \bm \sigma_Y}_{a \text{ times}}) \otimes (\underbrace{\bm \sigma_Z \otimes \dots \otimes \bm \sigma_Z}_{b\text{ times}}),
	\]
	and define $\bm \sigma'$ analogously. Note that $\bm \sigma$ and $\bm \sigma'$ are independent and identically distributed, and that
	\[
		\inn{\bm \sigma, \bm \sigma'} = \inn{\bm \sigma_Y, \bm \sigma_Y'}^a \inn{\bm \sigma_Z, \bm \sigma_Z'}^b,
	\]
	hence the result follows from \cref{obs:expected inner prod}.
\end{proof}

So all that remains is to pick a function $f:\R^2 \to \R$ which is non-positive on the set $\{(y,z) : y \leq -1 \text{ or } z \leq -1\}$, which does not grow too quickly, and which is analytic with non-negative Taylor coefficients. The following lemma, whose proof is a simple calculus exercise that we omit, provides such a function.
\begin{lemma}
	The function $f:\R^2 \to \R$ defined by
	\begin{equation}\label{eq:f def}
		f(y,z) \coloneqq 1+ y (2+\cosh \sqrt {2z}) + z (2+\cosh \sqrt{2y})
	\end{equation}
	is non-positive on the set $\{(y,z) : y \leq -1 \text{ or } z \leq -1\}$, is analytic with non-negative Taylor coefficients, has $f(0,0)=1$, and satisfies
	\begin{equation}\label{eq:f UB}
		f(y,z) \leq e^{3\sqrt{\max\{y,z\}+1}}
	\end{equation}
	for all $y,z \geq -1$.
\end{lemma}
With this, we have all of the ingredients in place to prove \cref{lem:geometric}, following the strategy outlined above.
\begin{proof}[Proof of \cref{lem:geometric}]
	We define
	\begin{equation*}
		\bm V \coloneqq f(\inn{\sigma_Y(\bm v), \sigma_Y(\bm w)}, \inn{\sigma_Z(\bm v), \sigma_Z(\bm w)})
	\end{equation*}
	as in \eqref{eq:V def}, where $f: \R^2 \to \R$ is given by \eqref{eq:f def}. By \cref{lem:moments positive}, the fact that $f$ is analytic with non-negative Taylor coefficients, and the fact that $f(0,0)=1$, we conclude that $\E[\bm V] \geq 1$. 

	Let $\cE$ be the event defined in \eqref{eq:E def}, and let $\overline \cE$ be its complement.
	The fact that $f$ is non-positive on the set $\{(y,z) : y \leq -1 \text{ or } z \leq -1\}$ implies that on the event $\overline \cE$, the random variable $\bm V$ is non-positive. Therefore $\E[\bm V \bm 1_{\overline{\cE}}] \leq 0$, and hence
	\[
		1 \leq \E[\bm V] = \E[\bm V \bm 1_{\cE}] + \E[\bm V \bm 1_{\overline{\cE}}] \leq \E[\bm V \bm 1_{\cE}].
	\]
	By \cref{obs:fubini}, we now have
	\begin{align*}
		1 \leq \E[\bm V\bm 1_{\cE}] &\leq \int_0^\infty \pr(\bm V\bm 1_{\cE} \geq \lambda)\dd \lambda=\int_0^\infty \pr(\{\bm V \geq \lambda\} \cap \cE)\dd \lambda.
	\end{align*}
	On the interval $[0,1]$, the integrand is at most $\pr(\cE)$, so
	\[
		1 \leq  \pr(\cE) + \int_1^\infty \pr(\{\bm V \geq \lambda\} \cap \cE) \dd \lambda.
	\]
	We now recall \eqref{eq:f UB}, which implies that on the event $\cE$, we have $\bm V \leq e^{3\sqrt{\bm M+1}}$, hence
	\[
		1- \pr(\cE) \leq \int_1^\infty \pr(\{\bm V \geq \lambda\} \cap \cE) \dd \lambda \leq \int_1^\infty \pr\left(\left\{e^{3\sqrt{\bm M+1}} \geq \lambda\right\} \cap \cE\right) \dd \lambda.
	\]
	We now change variables to $\lambda = e^{3\kappa}$, so that $e^{3\sqrt{\bm M+1}}\geq \lambda$ if and only if $\bm M \geq \kappa^2-1$, and conclude that
	\begin{equation}\label{eq:kappa integral}
		1 - \pr(\cE) \leq \int_{0}^\infty \pr(\{\bm M \geq \kappa^2-1\} \cap \cE) \cdot 3e^{3\kappa} \dd \kappa.
	\end{equation}
	We now claim that for some $\kappa \geq 0$, we have
	\begin{equation}\label{eq:prob LB}
		\pr(\{\bm M \geq \kappa^2-1\}\cap \cE) \geq \frac{e^{-4\kappa}}{4}.
	\end{equation}
	Indeed, if this does not hold, then continuing \eqref{eq:kappa integral}, we conclude that
	\[
		1-\pr(\cE) \leq \int_0^\infty \frac{e^{-4\kappa}}{4}\cdot 3e^{3\kappa} \dd \kappa = \frac 34 \int_0^\infty e^{-\kappa}\dd \kappa = \frac 34,
	\]
	hence $\pr(\cE) \geq \frac 14$. But on the event $\cE$ we have $\bm M \geq -1$, hence \eqref{eq:prob LB} holds with $\kappa=0$, a contradiction. We conclude that, as claimed, \eqref{eq:prob LB} holds for some $\kappa \geq 0$.

	Fixing such a $\kappa$, recalling the definitions of $\bm M$ and $\cE$ from \eqref{eq:M def} and \eqref{eq:E def}, and applying the union bound, we find that
	\begin{align*}
		\frac{e^{-4\kappa}}4 &\leq \pr(\{\bm M \geq \kappa^2-1\}\cap \cE)\\
		&=\pr\big(\max \{\inn{\sigma_Y(\bm v), \sigma_Y(\bm w)}, \inn{\sigma_Z(\bm v), \sigma_Z(\bm w)}\} \geq \kappa^2-1 \text{ and }\\
		&\hspace{3cm}\inn{\sigma_Y(\bm v), \sigma_Y(\bm w)}\geq -1 \text{ and }\inn{\sigma_Z(\bm v), \sigma_Z(\bm w)}\geq -1\big)\\
		&\leq \pr\big(\inn{\sigma_Y(\bm v), \sigma_Y(\bm w)} \geq \kappa^2-1 \text{ and } \inn{\sigma_Z(\bm v), \sigma_Z(\bm w)} \geq -1\big)\\
		&\hspace{3cm}+\pr\big(\inn{\sigma_Z(\bm v), \sigma_Z(\bm w)} \geq \kappa^2-1 \text{ and } \inn{\sigma_Y(\bm v), \sigma_Y(\bm w)} \geq -1\big).
	\end{align*}
	Therefore, either
	\[
		\pr\big(\inn{\sigma_Y(\bm v), \sigma_Y(\bm w)} \geq \kappa^2-1 \text{ and } \inn{\sigma_Z(\bm v), \sigma_Z(\bm w)} \geq -1\big) \geq \frac{e^{-4\kappa}}{8}
	\]
	or the same holds upon reversing the roles of $Y$ and $Z$. This concludes the proof of \cref{lem:geometric} upon recalling that $c=1/8$ and $2^{-C} = 1/64 \leq e^{-4}$.
\end{proof}

\section{Epilogue: Ramsey's original proof of Theorem \ref{thm:ramsey}}\label{sec:epilogue}
As mentioned in \cref{sec:intro}, there is no known proof of \cref{thm:ramsey} that does not use book graphs in some way. As a hopefully fitting end to this expos\'e, let us see the original proof of \textcite{MR1576401}, which uses book graphs in a rather different way from \cref{lem:book to clique}, yet whose proof shares certain ideas with the ones we have already seen.

Let us denote by $r(B_{t,m})$ the least integer $N$ such that every two-coloring of $E(K_N)$ contains a monochromatic copy of $B_{t,m}$. \textcite{MR1576401} proved the following upper bound on $r(B_{t,m})$.
\begin{theorem}[\cite{MR1576401}]\label{thm:ramsey book}
	$r(B_{t,m}) \leq (t+1)!\cdot m$ for all $t,m \geq 1$.
\end{theorem}
Note that since $K_k = B_{k-1,1}$, this immediately implies the bound $r(k) \leq k!$, and hence yields a proof of \cref{thm:ramsey}.
\begin{proof}[Proof of \cref{thm:ramsey book}]
	We proceed by induction on $t$. For the base case $t=1$, we wish to prove that $r(B_{1,m})\leq 2m$, which is immediate: in any two-coloring of $E(K_{2m})$, any vertex is incident to $2m-1$ edges, at least $m$ of which must have the same color by the pigeonhole principle. This yields a monochromatic copy of $B_{1,m}$.

	For the inductive step, suppose the result has been proved for $t-1$, and fix a coloring of $E(K_N)$ where $N = (t+1)!\cdot m = t!\cdot ((t+1)m)$. By the inductive hypothesis, this coloring contains a monochromatic copy of $B_{t-1,(t+1)m}$, which we may assume to be red without loss of generality. That is, there exist disjoint sets $A,X\subseteq V(K_N)$ with $\ab A = t-1$ and $\ab X = (t+1)m$, such that all edges inside $A$ and between $A$ and $X$ are red. If there is a vertex $v \in X$ with at least $m$ red neighbors in $X$, we may perform a ``red step'' by updating $A \to A \cup \{v\}$ and $X \to X \cap N_R(v)$, yielding a red $B_{t,m}$. So we may assume that every vertex in $X$ has at most $m-1$ red neighbors in $X$. Let $v_1$ be an arbitrary vertex of $X$, and let $X_1=N_B(v_1) \cap X$, so that $\ab{X_1} \geq \ab X-m$. Let $v_2$ be an arbitrary vertex of $X_1$, and let $X_2 = N_B(v_2) \cap X_1$, so that $\ab{X_2} \geq \ab{X_1}-m \geq \ab X - 2m$. Continuing in this way, we may build a sequence of vertices $v_1,\dots,v_t$ as well as a set $X_t$ with $\ab {X_t}\geq \ab X - tm = m$, such that each $v_i$ is adjacent in blue to all $v_j$ with $j>i$, as well as to all vertices in $X_t$. But this precisely means that we have constructed a blue $B_{t,m}$, completing the inductive step.
\end{proof}
Given \cref{thm:ramsey book}, it is natural to wonder what the true value of $r(B_{t,m})$ is. This question was first explicitly raised by \textcite{MR479691,MR679211}, who independently proved the bounds $(2^t-o(1))m\leq r(B_{t,m})\leq 4^t m$. Thomason in particular was motivated by \cref{lem:book to clique}, as discussed in \cref{sec:enter book}, and made the following bold conjecture.
\begin{conjecture}[\cite{MR679211}]\label{conj:thomason book}
	$r(B_{t,m}) \leq 2^t(m+t-2)+2$ for all $m,t \geq 1$.
\end{conjecture}
This conjecture is known to be true (and optimal) for $t \in \{1,2\}$, but it is wide open for $t \geq 3$ (and may well be false). Moreover, \cref{conj:thomason book} is likely to be very difficult: even the $m=1$ case would yield $r(k) \leq 2^{k+o(k)}$, a bound far stronger than anything currently known. However, a beautiful result of \textcite{MR4115773} confirms this conjecture asymptotically for any fixed $t$.
\begin{theorem}[\cite{MR4115773}]
	$r(B_{t,m})=(2^t+o(1))m$ as $m \to \infty$, for any fixed $t$.
\end{theorem}
Conlon's result addresses a question that arguably goes back 90 years to the original work of \textcite{MR1576401}, yet uses highly sophisticated tools developed in the interim, such as the regularity lemma of \textcite{MR0540024}, and this question is also closely related to a famous conjecture of \textcite{MR0371701}, which was recently resolved by \textcite{MR3664811}. In fact, Ramsey theory has seen a number of recent breakthroughs on old, seemingly intractable problems by the introduction of remarkable new techniques: three other examples from the past two years are the works of \textcite{MR4720317} on explicit constructions, of \textcite{MR4713025} on off-diagonal Ramsey numbers, and of \textcite{2308.15589} on restricted Ramsey graphs. There is every reason to hope and expect this trend to continue.

\printbibliography

\end{document}
